\documentclass[11pt, twoside, a4paper, reqno]{amsart}

\usepackage{amsfonts, amssymb, amscd, amsmath, amsthm}
\usepackage[utf8]{inputenc}
\usepackage{graphicx}
\usepackage{float}
\usepackage{multicol}
\allowdisplaybreaks 
\usepackage{mathtools}
\usepackage{amsmath, amssymb}
\usepackage{amscd}
\usepackage{xcolor}
\usepackage{tikz-cd} 
\usepackage{graphicx} 
\usepackage[charter]{mathdesign}
\usepackage{enumitem}

\allowdisplaybreaks 
\usepackage{mathtools}
\usepackage[bookmarksnumbered, colorlinks, plainpages]{hyperref}

\usepackage{comment}
\newtheorem{definition}{Definition}[section]
\newtheorem{theorem}[definition]{Theorem}
\newtheorem{lemma}[definition]{Lemma}

\newtheorem{proposition}[definition]{Proposition}
\theoremstyle{definition}
\newtheorem{remark}[definition]{Remark}
\newtheorem{note}[definition]{Note}
\newtheorem{example}[definition]{Example}
\newtheorem{notation}[definition]{Notation}

\newcommand{\diX}{\int^\oplus_{X}}
\newcommand{\dilX}{\int^{\oplus_\text{loc}}_{X_{\text{loc}}}}
\newcommand{\dmu}{\mathrm{d} \mu (p)}
\newcommand{\la}{\left\langle}
\newcommand{\ra}{\right\rangle}
\newcommand\up[1]{\mbox{\raisebox{1pt}{\ensuremath{#1}}}} 

\title[Direct Integral of Locally Hilbert Spaces Over a Locally Measure Space]{Direct Integral of Locally Hilbert Spaces Over a Locally Measure Space}
\author[Kulkarni]{Chaitanya J. Kulkarni}
\address{Chaitanya J. Kulkarni, Indian Institute of Science  Education and Research (IISER) Mohali, Knowledge City, S.A.S Nagar, Punjab 140306, India.}
\email{chaitanyakulkarni58@gmail.com}

\author[Pamula]{Santhosh Kumar Pamula}
\address{Santhosh Kumar Pamula, Indian Institute of Science  Education and Research (IISER) Mohali, Knowledge City, S.A.S Nagar, Punjab 140306, India.}
\email{santhoshkp@iisermohali.ac.in}

\date{}
\begin{document}

\maketitle

\begin{abstract}
In this work, we introduce the concept of the direct integral of locally Hilbert spaces by generalizing the classical notion of a measure space to that of a locally measure space. We establish that the direct integral of a family of locally Hilbert spaces over a locally measure space forms a locally Hilbert space. We then define two important subclasses of locally bounded operators on such direct integrals, namely  decomposable locally bounded operators and diagonalizable locally bounded operators. We show that each of these subclasses forms a locally von Neumann algebra, and in particular, that the locally von Neumann algebra of diagonalizable operators is abelian. Finally, we prove that the locally von Neumann algebra of diagonalizable operators coincides with the commutant of the locally von Neumann algebra of decomposable operators.
   

\end{abstract}


\section{Introduction} \label{sec;Introduction}
The concept of the direct integral of Hilbert spaces extends the idea of the direct sum of Hilbert spaces by replacing the discrete index set used in direct sums with an appropriate measure space. The notion of the direct integral of Hilbert spaces corresponds to an abelian von Neumann algebra, known as the algebra of diagonalizable operators. Conversely, given a separable Hilbert space $\mathcal{H}$ and an abelian von Neumann algebra in the algebra of bounded operators $\mathcal{B}(\mathcal{H})$, one can obtain a measure space and a family of separable Hilbert spaces such that the given Hilbert space can be identified with the direct integral of Hilbert spaces, whereas the abelian von Neumann algebra can be identified with the algebra of diagonalizable operators (see \cite[Part II, Chapter 6, Theorem 2]{DixV}). This process is known as the disintegration of the Hilbert space. The disintegration of Hilbert spaces (see \cite{DJ, Wils1, Wils2}) plays a fundamental role in various areas of operator algebras, particularly in decomposing a representation of a separable $C^\ast$-algebra into irreducible components, or in decomposing a von Neumann algebra into factors  (see \cite{OB1, DixC, DixV, KR2, Tak1} for further details on this topic). 

In \cite{AD}, the concept of the (finite) direct sum of locally Hilbert spaces was introduced. In this article, we extend this notion by replacing the discrete measure space with a more general structure, referred to as a locally measure space. Based on this framework, we define the direct integral of locally Hilbert spaces over a locally measure space. In the special case where the underlying measure space is discrete, our construction coincides with the classical direct sum. 

To develop the concept of the direct integral of locally Hilbert spaces, we begin by introducing the notion of a locally measure space, which is defined by a projective system of measures over an inductive system of measurable spaces (see Definition \ref{def; lms}).  Using this framework, we define the direct integral of locally Hilbert spaces over a locally measure space and prove that it is again a locally Hilbert space. It is important to note that the structure of the locally measure space plays a crucial role, even when each locally Hilbert space is a Hilbert space, their direct integral over a locally measure space need not be a Hilbert space. 
Since the direct integral of locally Hilbert spaces over a locally measure space is a locally Hilbert space, we are naturally led to define two important classes of locally bounded operators: decomposable and diagonalizable locally bounded operators. These notions are inspired by the decomposable and diagonalizable operators on the direct integrals of Hilbert spaces. We show that each of these classes forms a locally von Neumann algebra, which we see by showing it as the projective limit of a projective system of von Neumann algebras. In particular, the locally von Neumann algebra of diagonalizable locally bounded operators is shown to be abelian. Furthermore, in certain cases, we establish that the locally von Neumann algebra of decomposable locally bounded operators coincides with the commutant of the locally von Neumann algebra of diagonalizable locally bounded operators. Throughout this article, we make use of inductive and projective limits of locally convex spaces, combined with the theory of direct integrals. For this purpose, we refer to the results presented in \cite{BGP, AD, AG, AI, MJ3, NCP}. We also employ techniques related to the theory of locally von Neumann algebras developed in \cite{MF, MJ1, MJ2}.

This article is divided into four sections. In Section \ref{sec;Preliminaries} we review key definitions and results from the theories of direct integrals of Hilbert spaces, locally $C^\ast$-algebras and locally von Neumann algebras. In Section \ref{sec; Direct integrals}, we introduce the concept of a locally measure space (Definition \ref{def; lms}) and provide an illustrative example (Example \ref{ex;lsms}). We then define the direct integral of locally Hilbert spaces over such spaces and show that it forms a locally Hilbert space (Proposition \ref{prop;dilhs}). Several examples are presented. In Example \ref{ex;dilhs direct sum and direct integral}, we show that this construction coincides with the usual direct sum when the underlying measure space is discrete. Furthermore, Examples \ref{ex;dilhs direct integral of C} and \ref{ex;dilhs direct integral of L2 R mu} illustrate that, even if each locally Hilbert space is a Hilbert space, their direct integral over a locally measure space need not be a Hilbert space. However, it forms a locally Hilbert space. In Section \ref{sec; Decomposable and Diagonalizable Locally Bounded Operators}, we introduce and study two new classes of locally bounded operators on the direct integral of locally Hilbert spaces, namely decomposable locally bounded operators and diagonalizable locally bounded operators (see Definition \ref{def;DecDiag(lbo)}). We provide examples for both classes. By definition, every diagonalizable locally bounded operator is decomposable. However, we give an example of a decomposable locally bounded operator that is not  diagonalizable (Example \ref{eg; Dec but not Diag 1} and \ref{eg; Dec but not Diag 2}). We also present an example of locally bounded operators that is not decomposable (Example \ref{eg; LBO but not Dec}). In Theorem \ref{thm;DEC and DIAG LvNA}, we prove that each of these classes forms a locally von Neumann algebra. Finally, we study the relationship between these locally von Neumann algebras and establish that they are commutant of each other in certain specific cases (Theorem \ref{thm; M DEC = M DIAG Commutant}).

\section{Preliminaries} \label{sec;Preliminaries}

\subsection{Direct integral of Hilbert spaces}
We begin by recalling a few fundamental definitions and results from the theory of direct integrals of Hilbert spaces. For a comprehensive treatment of this subject, the reader is referred to \cite{OB1, DixC, DixV, KR2, Tak1}. In particular, throughout this article, we follow the terminology and notational conventions presented in Chapter 14 of \cite{KR2}. 

\begin{definition}\cite[Definition 14.1.1]{KR2} \label{def;dihs}
If $X$ is a $\sigma$-compact locally compact (Borel measure) space, $\mu$ is the completion of a (positive) Borel measure on $X$, and $\{ \mathcal{H}_p \}_{p \in X}$ is a family of separable Hilbert spaces indexed by points $p$ in $X$, we say that a separable Hilbert space $\mathcal{H}$ is the direct integral of $\{ \mathcal{H}_p \}$ over $(X, \mu)$ \Big(we write: $\mathcal{H} = \diX \mathcal{H}_p \, \dmu$ \Big) when, to each $x \in \mathcal{H}$, there corresponds a function $ p \mapsto x(p)$ on $X$ such that $x(p) \in \mathcal{H}_p$ for each $p$ and 
\begin{enumerate}
\item $p \mapsto \la x(p), y(p) \ra$ is $\mu$-integrable, when $x, y \in \mathcal{H}$ and
\begin{align*}
\la x, y \ra = \int_X \la x(p), y(p) \ra \, \dmu
\end{align*}
\item if $x_p \in \mathcal{H}_p$ for all $p$ in $X$ and $p \mapsto \la x_p, y(p) \ra$ is integrable for each $y \in \mathcal{H}$, then there is a $x \in \mathcal{H}$ such that 
$x(p) = x_p$ for a.e. $p \in X$. 
\end{enumerate}
We say that $\diX \mathcal{H}_p \, \dmu$ and $p \mapsto x(p)$ are the (direct integral) decompositions of $\mathcal{H}$ and $x$ respectively.
\end{definition}


Given a measure space $(X, \mu)$, we denote by $\text{L}^\infty(X, \mu)$ the space of all essentially bounded measurable functions from $X$ to $\mathbb{C}$. We now recall the definitions of decomposable and diagonalizable bounded operators on the direct integral of Hilbert spaces.
\begin{definition}\cite[Definition 14.1.6]{KR2} \label{def;Debo}
Let $(X, \mu)$ be a measure space, and let  $\{ \mathcal{H}_p \}_{p \in X}$ be a family of separable Hilbert spaces with $\mathcal{H} = \diX \mathcal{H}_p \, \dmu$.
\begin{enumerate}
    \item \label{def;Decbo} An operator $T$ in $\mathcal{B}(\mathcal{H})$ is said to be {\it decomposable}, if there is a family  $\{ T_p \in \mathcal{B}(\mathcal{H}_p) \}_{p \in X}$ such that for each $x \in \mathcal{H}$, we have 
\begin{align*}
(Tx)(p) = T_p x(p) \; \; \; \text{for a.e.} \; \; p \in X.
\end{align*}
Subsequently, $T$ is denoted by $\int^\oplus_X T_p \, \dmu$. Moreover, the norm of $T$ is given by
\begin{equation} \label{eq;norm of T}
\| T \| = \text{ess sup} \big \{ \| T_p \| \; : \; p \in X \big \}.
\end{equation}
\item \label{def;Diagbo} An operator $T$ in $\mathcal{B}(\mathcal{H})$ is said to be {\it diagonalizable}, if $T$ is decomposable and there exists a function $f \in \text{L}^\infty(X, \mu)$ such that for each $x \in \mathcal{H}$, we have 
\begin{align*}
(Tx)(p) = f(p) x(p) \; \; \; \text{ for a.e.} \; \; p \in X.
\end{align*}
\end{enumerate} 
\end{definition}

The following theorem describes the structure of the set of all decomposable and diagonalizable operators on a direct integral of Hilbert spaces, and describes the relationship between these two classes of operators.

\begin{theorem}\cite[Theorem 14.1.10]{KR2} \label{thm;DeDibo vNA}
Let $\mathcal{H} = \diX \mathcal{H}_p \, \dmu$  be as in Definition \ref{def;Debo}. Then the set of all decomposable operators is a von Neumann algebra with the abelian commutant coinciding with the set of all diagonalizable operators.  
\end{theorem}

In this work, our goal is to develop appropriate notions such as direct integrals, decomposable operators, diagonalizable operators, and related constructions within the framework of locally Hilbert spaces. Before proceeding to these definitions, we first review the basic notations, terminology, and fundamental concepts from the theories of locally Hilbert spaces, locally C*-algebras and locally von Neumann algebras. The following definitions and ideas are primarily based on \cite{AG, MJ1}. For a more detailed exposition of these topics, the reader is referred to \cite{BGP, AD, MF, AG, AG2, AG3, AI, MJ1, MJ2, MJ3, NCP}.

\subsection{Locally Hilbert space} 
A locally Hilbert space is defined as the inductive limit of a strictly inductive system (also referred to as an upward filtered family) of Hilbert spaces. The formal definition is given below.

\begin{definition}\cite[Subsection 1.3]{AG} \label{def;sis}
Let $( \mathcal{H}_\alpha,\; \la \cdot, \cdot \ra_{\mathcal{H}_\alpha} )_{\alpha \in \Lambda}$ be a net of Hilbert spaces. Then $\mathcal{F} = \{\mathcal{H}_\alpha \}_{\alpha \in \Lambda}$ is said to be a strictly inductive system (or an upward filtered family) of Hilbert spaces if:
\begin{enumerate}
\item $(\Lambda, \leq)$ is a directed partially ordered set (poset);
\item for each $\alpha, \beta \in \Lambda$ with $\alpha \leq \beta \in \Lambda$ we have $\mathcal{H}_\alpha \subseteq \mathcal{H}_\beta$;
\item for each $\alpha, \beta \in \Lambda$ with $\alpha \leq \beta \in \Lambda$ the inclusion map $J_{\beta, \alpha} : \mathcal{H}_\alpha \rightarrow \mathcal{H}_\beta$ is isometric, that is,
$\la u, v \ra_{\mathcal{H}_\alpha} = \la u, v \ra_{\mathcal{H}_\beta}$ for all $u, v \in \mathcal{H}_\alpha$.
\end{enumerate}
\end{definition}

As shown in Equation (1.13) of \cite{AG}, for a strictly inductive system $\mathcal{F} = \{\mathcal{H}_\alpha \}_{\alpha \in \Lambda}$ of Hilbert spaces, the inductive limit denoted by $\varinjlim\limits_{\alpha \in \Lambda} \mathcal{H}_\alpha$ exists, and it is given by
\begin{equation} \label{eq;lhs}
\varinjlim_{\alpha \in \Lambda} \mathcal{H}_\alpha = \bigcup_{\alpha \in \Lambda} \mathcal{H}_\alpha.
\end{equation}

\begin{definition}\cite[Subsection 1.3]{AG} \label{def;lhs}
A locally Hilbert space $\mathcal{D}$ is defined as the inductive limit of some strictly inductive system $\mathcal{F} = \{\mathcal{H}_\alpha \}_{\alpha \in \Lambda}$ of Hilbert spaces. 
\end{definition}

It is worth noting that in the works \cite{BGP, AD}, the authors use the term quantized domain in place of what we refer to as a locally Hilbert space. In particular, the article \cite{BGP} adopts the notation $\{\mathcal{H}; \mathcal{F}; \mathcal{D} \}$ to denote a quantized domain, where $\mathcal{H}$ is the Hilbert space obtained as the completion of the locally Hilbert space $\mathcal{D}$ which is obtained by the inductive limit of the strictly inductive system $\mathcal{F}$. On the other hand, in \cite{MJ3}, the author uses the notation $\mathcal{D}_\mathcal{F}$. Throughout this work, we adopt the notation 
\begin{equation} \label{eq; quantized domain}
\big \{ \mathcal{H}; \mathcal{F} = \{\mathcal{H}_\alpha\}_{\alpha \in \Lambda}; \mathcal{D} \big \}    
\end{equation}
to represent a quantized domain, where the Hilbert space $\mathcal{H}$ is the completion of the locally Hilbert space $\mathcal{D}$ given by the inductive limit of the strictly inductive system  $\mathcal{F} = \{\mathcal{H}_\alpha\}_{\alpha \in \Lambda}$. Thus, we have 
\begin{equation*}
\mathcal{D} = \varinjlim_{\alpha \in \Lambda} \mathcal{H}_\alpha = \bigcup_{\alpha \in \Lambda} \mathcal{H}_\alpha \; \; \; \; \text{and} \; \; \; \; \overline{\mathcal{D}} = \mathcal{H}.
\end{equation*}
Moreover, if $\Lambda = \mathbb{N}$, then $\mathcal{E}$ is countable and in this case $\big \{ \mathcal{H}; \mathcal{F} = \{\mathcal{H}_\alpha\}_{\alpha \in \Lambda}; \mathcal{D} \big \}$ is called a \textit{quantized Fr\'echet domain}. For details, see Definition 2.3 of \cite{BGP}.


\begin{example} \label{ex;lhs}
Let $e_n$ be a sequence in $\ell^2(\mathbb{N})$, where the $n^\text{th}$ term of the sequence $e_n$ is 1 and 0 elsewhere. Then $\{e_{n}: n \in \mathbb{N}\}$ is a Hilbert basis of $\ell^2(\mathbb{N}).$ For each $k \in \mathbb{N}$, define the closed (in fact, finite dimensional) subspace  $\mathcal{H}_k := \text{span} \{ e_1, e_2,...,e_k \}$. It follows that $\mathcal{H}_{m} \subseteq \mathcal{H}_{n}$ for $m \leq n$. In other words, the family $\mathcal{F} = \big \{\mathcal{H}_n \big \}_{n \in \mathbb{N}}$ forms a strictly inductive system of Hilbert spaces. The inductive limit $\mathcal{D}$ of the strictly inductive system $\mathcal{F} = \big \{\mathcal{H}_n \big \}_{n \in \mathbb{N}}$ is given by $\mathcal{D} = \varinjlim\limits_{n \in \mathbb{N}} \mathcal{H}_n = \bigcup\limits_{n \in \mathbb{N}} \mathcal{H}_n$, and thus
\begin{align*}
\mathcal{D} = \big \{ x = (x_1, x_2, \cdots, ) \in \ell^2(\mathbb{N}) \; : \; x_n =0 \;\text{for all but finitely many}  \; n \in \mathbb{N} \big \}.
\end{align*}
Here $\mathcal{D}$ is the locally Hilbert space and $\overline{\mathcal{D}} = \ell^2(\mathbb{N})$. In other words, $\big \{ \ell^2(\mathbb{N}), \mathcal{F} = \{\mathcal{H}_n\}_{n \in \mathbb{N}}, \mathcal{D} \big \}$ is a quantized Frechet domain. For a more general construction, one can see Example 2.9 of \cite{BGP}. 
\end{example}

For the remainder of this section, unless stated otherwise, we assume that $(\Lambda, \leq)$ is a directed poset. We now recall the notion of a locally bounded operator between locally Hilbert spaces. Let $\big \{ \mathcal{H}; \mathcal{F} = \{\mathcal{H}_\alpha\}_{\alpha \in \Lambda}; \mathcal{D} \big \}$  and $\big \{ \mathcal{K}; \mathcal{G} = \{\mathcal{K}_\alpha\}_{\alpha \in \Lambda}; \mathcal{O} \big \}$  be two quantized domains. Consider the families $\{ P_\alpha : \mathcal{H} \rightarrow \mathcal{H}_\alpha \}_{\alpha \in \Lambda}$ and $\{ Q_\alpha : \mathcal{K} \rightarrow \mathcal{K}_\alpha\}_{\alpha \in \Lambda}$ of orthogonal projections. Suppose $\mathcal{L}(\mathcal{D}, \mathcal{O})$ denotes the collection of all densely defined (in $\mathcal{H}$) operators from $\mathcal{D}$ to $\mathcal{O}$, and in particular, $\mathcal{L}(\mathcal{D}, \mathcal{D}) = \mathcal{L}(\mathcal{D})$. Let us consider a subclass of $\mathcal{L}(\mathcal{D}, \mathcal{O})$ defined as   
\begin{equation*}
C_{\mathcal{F}, \mathcal{G}}(\mathcal{D}, \mathcal{O}) := \big \{ T \in \mathcal{L}(\mathcal{D}, \mathcal{O}) \; \; : \; \; TP_\alpha = Q_\alpha T P_\alpha \in \mathcal{B}(\mathcal{H}, \mathcal{K}) \; \; \text{for each} \; \alpha \in \Lambda \big \}.
\end{equation*}
Note that, if $T \in \mathcal{L}(\mathcal{D}, \mathcal{O})$, then $T \in C_{\mathcal{F}, \mathcal{G}}(\mathcal{D}, \mathcal{O})$ if and only if $T(\mathcal{H}_{\alpha}) \subseteq \mathcal{K}_{\alpha}$ and $T\big|_{\mathcal{H}_{\alpha}} \in \mathcal{B}(\mathcal{H}_{\alpha}, \mathcal{K}_\alpha)$ for each $\alpha \in \Lambda$ (refer \cite[Section 2.2]{GIS}). In particular, $C_{\mathcal{F}, \mathcal{F}}(\mathcal{D}, \mathcal{D}) = C_{\mathcal{F}}(\mathcal{D})$ (see \cite[Equation (3.1)]{AD}).  Consider a subclass of $C_{\mathcal{F}, \mathcal{G}}(\mathcal{D}, \mathcal{O})$ given by
\begin{equation*}
C^\ast_{\mathcal{F}, \mathcal{G}}(\mathcal{D}, \mathcal{O}) := \big \{ T \in C_{\mathcal{F}, \mathcal{G}}(\mathcal{D}, \mathcal{O}) \; \; : \; \; Q_\alpha T \subseteq T P_\alpha \; \; \text{for each} \; \alpha \in \Lambda \big \}.    
\end{equation*}
By following \cite[Section 2.1]{BGP}, we get that an operator $T \in \mathcal{L}(\mathcal{D}, \mathcal{O})$ belongs to $C^\ast_{\mathcal{F}, \mathcal{G}}(\mathcal{D}, \mathcal{O})$ if and only if $T(\mathcal{H}_{\alpha}) \subseteq \mathcal{K}_{\alpha}$,  $T(\mathcal{H}_{\alpha}^{\bot}\cap \mathcal{D}) \subseteq \mathcal{K}_{\alpha}^{\bot}\cap \mathcal{O}$ and  $T\big|_{\mathcal{H}_{\alpha}} \in \mathcal{B}(\mathcal{H}_{\alpha}, \mathcal{K}_\alpha)$ for each $\alpha \in \Lambda$ (also, see \cite[Section 2.2]{GIS}).

\begin{definition}\cite[Section 2.2]{GIS} \label{def;lbo}
Let  $\big \{ \mathcal{H}; \mathcal{F} = \{\mathcal{H}_\alpha\}_{\alpha \in \Lambda}; \mathcal{D} \big \}$  and $\big \{ \mathcal{K}; \mathcal{G} = \{\mathcal{K}_\alpha\}_{\alpha \in \Lambda}; \mathcal{O} \big \}$  be two quantized domains. Then an operator $T \in \mathcal{L}(\mathcal{D}, \mathcal{O})$ is said to be locally bounded if $T \in C^\ast_{\mathcal{F}, \mathcal{G}}(\mathcal{D}, \mathcal{O})$.
\end{definition}

\noindent
Now, we give an attention to $C^\ast_{\mathcal{F}, \mathcal{F}}(\mathcal{D}, \mathcal{D}) = C^\ast_{\mathcal{F}}(\mathcal{D})$ which can be seen as a (proper) subclass of $\mathcal{L}(\mathcal{D})$. Suppose $T \in C^\ast_\mathcal{F}(\mathcal{D})$, then $T$ has an unbounded dual $T^\bigstar$ satisfying  \begin{equation*}
\mathcal{D} \subseteq \text{dom}(T^\bigstar), \;  \;  \; T^\bigstar(\mathcal{D}) \subseteq \mathcal{D}.
\end{equation*}
If we define $T^\ast := T^\bigstar\big|_{\mathcal{D}}$, then it follows that the correspondence $T \mapsto T^\ast$ is an involution on $C^\ast_\mathcal{F}(\mathcal{D})$. Conversely, if $T \in \mathcal{L}(\mathcal{D})$  with
\begin{equation} \label{eqn; locally bounded operator on D}
T(\mathcal{H}_\alpha) \subseteq \mathcal{H}_\alpha; \; \; \; T^\ast(\mathcal{H}_\alpha) \subseteq \mathcal{H}_\alpha; \; \; \; T\big|_{\mathcal{H}_\alpha}, T^\ast\big|_{\mathcal{H}_\alpha} \in \mathcal{B}(\mathcal{H}_\alpha) \; \; \text{for each} \; \; \alpha \in \Lambda,
\end{equation}
then $T \in C^\ast_\mathcal{F}(\mathcal{D})$ \big(see \cite[Proposition 3.1]{AD} \big). As a result $C^\ast_\mathcal{F}(\mathcal{D})$ is unital $\ast$-algebra.

\begin{example} \label{ex;lbo}
Let $\big \{ \ell^2(\mathbb{N}), \mathcal{F} = \{\mathcal{H}_n\}_{n \in \mathbb{N}}, \mathcal{D} \big \}$ be the quantized Fr\'echet domain as given in Example \ref{ex;lhs}. Now, consider the linear operator $S$ in $\ell^2(\mathbb{N})$ given by the matrix 
\begin{align*}
S= \begin{bmatrix}
1      & 0      & 0      & \cdots & 0      & \cdots \\
0      & 2      & 0      & \cdots & 0      & \cdots \\
0      & 0      & 3      & \cdots & 0      & \cdots \\
\vdots & \vdots & \vdots & \ddots & \vdots & \vdots\\
0      & 0      & 0      & \cdots & m      &  \cdots\\
0      & 0      & 0      & \cdots & 0      & \ddots
\end{bmatrix}.
\end{align*}
with the domain 
\begin{align*}
\text{dom}(S) = \Big \{ \big \{x_n \big\}_{n \in \mathbb{N}} \in \ell^2(\mathbb{N}) \;\; \big| \;\; \sum^\infty_{n =1} n^2 | x_n |^2 < \infty \Big \}.
\end{align*}
Since the locally Hilbert space $\mathcal{D}$ consists of sequences with finite support, this gives the inclusion $\mathcal{D} \subset \text{dom}(S)$. Here the inclusion is proper because the sequence $\big \{\frac{1}{n^2} \big\}_{n \in \mathbb{N}} \in \text{dom}(S) \setminus \mathcal{D}$. Since for each $n \in \mathbb{N}$, we have $Se_n = n e_n$, we obtain $S(\mathcal{H}_n) \subseteq \mathcal{H}_n$ and $S(\mathcal{H}_{n}^{\bot}\cap \mathcal{D}) \subseteq \mathcal{H}_{n}^{\bot}\cap \mathcal{D}$. Moreover, $\big \| S\big|_{\mathcal{H}_n} \big \| = n$ for all $n \in \mathbb{N}$, that is, $S\big|_{\mathcal{H}_n} \in \mathcal{B}(\mathcal{H}_n)$. It follows from Equation \eqref{eqn; locally bounded operator on D} that $T := S\big|_{\mathcal{D}} \in C^\ast_\mathcal{F}(\mathcal{D})$. Equivalently, $T$ is a locally bounded operator on $\mathcal{D}$. 
\end{example}

\begin{remark}\cite[Subection 1.4]{AG} \label{rem;lbo}
Let  $\big \{ \mathcal{H}; \mathcal{F} = \{\mathcal{H}_\alpha\}_{\alpha \in \Lambda}; \mathcal{D} \big \}$  and $\big \{ \mathcal{K}; \mathcal{G} = \{\mathcal{K}_\alpha\}_{\alpha \in \Lambda}; \mathcal{O} \big \}$  be two quantized domains. Consider $T \in C^*_{\mathcal{F}, \mathcal{G}}(\mathcal{D}, \mathcal{O})$, and for each $\alpha \in \Lambda$ take 
$T_\alpha := T\big|_{\mathcal{H}_{\alpha}}$. Then $T_\alpha = Q_{\mathcal{K}_\alpha}T\big|_{\mathcal{H}_{\alpha}}$, where $Q_{\mathcal{K}_\alpha}$ denotes the orthogonal projection of $\mathcal{O}$ to $\mathcal{K}_\alpha$ (see \cite[Lemma 3.1]{AG3} for the existence of such projection). For a fixed $\alpha \in \Lambda$, we denote the inclusion maps by the notations $J_{\mathcal{D}, \alpha} : \mathcal{H}_\alpha \rightarrow \mathcal{D}$ and $J_{\mathcal{O}, \alpha} : \mathcal{K}_\alpha \rightarrow \mathcal{O}$. Then the collection $\{ T_\alpha \}_{\alpha \in \Lambda}$ satisfies the following properties:
\begin{enumerate}
\item for each $\alpha \in \Lambda$, $T_\alpha \in \mathcal{B}(\mathcal{H}_\alpha, \mathcal{K}_\alpha)$ and $T J_{\mathcal{D}, \alpha} = J_{\mathcal{O}, \alpha} T_\alpha$;
\item for $\alpha, \beta \in \Lambda$ with $\alpha \leq \beta$, we get $T^*_\beta\big|_{\mathcal{K}_{\alpha}} = T^*_\alpha$.
\end{enumerate}
\end{remark}
In view of Remark \ref{rem;lbo} and following the notations of \cite{AG}, we say that every $T \in C^{\ast}_{\mathcal{F}}(\mathcal{D})$ can be seen as a projective (or inverse) limit of the net $\{T_{\alpha}\}_{\alpha \in \Lambda}$ of bounded operators. That is, 
\begin{equation} \label{eq; inverese limit of bounded operators}
    T = \varprojlim\limits_{\alpha \in \Lambda} T_{\alpha}.
\end{equation}
We are now in a position to introduce the notion of a locally $C^{\ast}$-algebra. For a comprehensive reading of these algebras and related concepts such as local completely positive maps, the reader is referred to \cite{BGP, AD, AG, AG2, AI, MJ1, MJ2, MJ3, NCP} and the references therein.
\subsection{Locally $C^{\ast}$-algebra} Let $\mathcal{A}$ be a unital $\ast$-algebra. A seminorm $p$ on $\mathcal{A}$ is said to be a $C^{\ast}$-seminorm, if 
\begin{equation*}
p(a^\ast) = p(a), \; \; \; \; p(a^*a) = p(a)^2, \; \; \; \; p(ab) \leq p(a) p(b),
\end{equation*}
\noindent for all $a,b \in \mathcal{A}$. It is important to note that in \cite{ZS}, the author proved that the condition $p(a^*a) = p(a)^2$ implies $p(ab) \leq p(a) p(b)$ for all $a,b \in \mathcal{A}$. Let $\mathcal{P} := \{ p_\alpha \;  : \;   \alpha \in \Lambda \}$ be a family of $C^{\ast}$-seminorms defined on a the $\ast$-algebra $\mathcal{A}$. Then $\mathcal{P}$ is called a upward filtered family if for each $a \in \mathcal{A}$, we have  $p_\alpha(a) \leq p_\beta(a)$, whenever $\alpha \leq \beta$.

\begin{definition} \cite{AI} \label{def;lca}
A $*$-algebra $\mathcal{A}$ which is complete with respect to the locally convex topology generated by an upward filtered family $\{ p_\alpha \;  : \;  \alpha \in \Lambda \}$ of C*-seminorms is called a locally $C^{\ast}$-algebra. Further, $\mathcal{A}$ is called a unital $*$-algebra, if $\mathcal{A}$ has unit.
\end{definition}

It is well known that every locally $C^{\ast}$-algebra can be realized as the projective limit (or inverse limit) of a projective system (or inverse system) of $C^{\ast}$-algebras. The construction of such a projective system is detailed in \cite{BGP, AG, MJ1, NCP}. For the reader's convenience, we briefly recall the essential steps below. Let $\mathcal{A}$ be a locally $C^{\ast}$-algebra. Then for each $\alpha \in \Lambda$, take $\mathcal{I}_\alpha := \{ a \in \mathcal{A} \; : \; p_\alpha(a) = 0 \}$, which is a two-sided closed ideal in $\mathcal{A}$, then $\mathcal{A}_\alpha : = {\mathcal{A}}\!/\!{\mathcal{I}_\alpha}$ is a $C^{\ast}$-algebra with respect to the $C^{\ast}$-norm induced by $p_\alpha$ (refer \cite{CA, NCP} for more details). Whenever $\alpha \leq \beta$, since $P_\alpha(a) \leq p_\beta(a)$ for each $a \in \mathcal{A}$, there is a $C^\ast$-homomorphism (surjective) $\pi_{\alpha, \beta} : \mathcal{A}_\beta \rightarrow \mathcal{A}_\alpha$ given by $\pi_{\alpha, \beta}(a + \mathcal{I}_\beta) = a + \mathcal{I}_\alpha$. Thus, the pair  $\left ( \{ \mathcal{A}_\alpha \}_{\alpha \in \Lambda},  \{ \pi_{\alpha, \beta} \}_{\alpha \leq \beta}\right )$ forms a projective system of $C^{\ast}$-algebras. The projective limit (or inverse limit) of this system is given by (see Subsection 1.1 of \cite{AG}):
\begin{align}\label{eq;pl}
\varprojlim\limits_{\alpha \in \Lambda} \mathcal{A}_\alpha := \left \{ \{x_\alpha\}_{\alpha \in \Lambda} \in \prod_{\alpha \in \Lambda} \mathcal{A}_\alpha ~~ : ~~ \pi_{\alpha, \beta}(x_\beta) = x_\alpha, ~~ \text{whenever} ~~ \alpha \leq \beta \right \},
\end{align}
which is equipped with the weakest locally convex topology that makes each linear map $\pi_\beta : \varprojlim\limits_{\alpha \in \Lambda} \mathcal{A}_\alpha \rightarrow \mathcal{A}_\beta$ defined by $\pi_\beta(\{x_\alpha\}_{\alpha \in \Lambda}) := x_\beta$ is a continuous $\ast$-homomorphism. This topology is known as the projective limit topology. Since $\mathcal{A}_\alpha$ is complete for each $\alpha \in \Lambda$, the projective limit $\varprojlim\limits_{\alpha \in \Lambda} \mathcal{A}_\alpha$ is complete with respect to the projective limit topology \big (see Subsection 1.1 of \cite{AG} \big ). Moreover, the pair $\big(\varprojlim\limits_{\alpha \in \Lambda} \mathcal{A}_\alpha, \{ \pi_\alpha \}_{\alpha \in \Lambda} \big)$ is compatible with the projective system $\left ( \{ \mathcal{A}_\alpha \}_{\alpha \in \Lambda},  \{ \pi_{\alpha, \beta} \}_{\alpha \leq \beta} \right )$, that is, $\pi_{\alpha, \beta} \circ \pi_\beta = \pi_\alpha$, whenever  $\alpha \leq \beta$. 

Let $\mathcal{W}$ be a locally convex $\ast$-algebra along with the net of continuous $\ast$-homomorphisms given by $\big \{\psi_\alpha : \mathcal{W} \rightarrow \mathcal{A}_\alpha \big\}_{\alpha \in \Lambda}$. If the pair $(\mathcal{W}, \{ \psi_\alpha \}_{\alpha \in \Lambda})$ is compatible with the projective system $\left ( \{ \mathcal{A}_\alpha \}_{\alpha \in \Lambda},  \{ \pi_{\alpha, \beta} \}_{\alpha \leq \beta} \right )$, that is, 
\begin{equation}\label{eq;ps}
\pi_{\alpha, \beta} \circ \psi_\beta = \psi_\alpha    \; \text{whenever} \; \alpha \leq \beta, 
\end{equation}
then there always exists a unique continuous linear map $\Psi : \mathcal{W} \rightarrow \varprojlim\limits_{\alpha \in \Lambda} \mathcal{A}_\alpha$ satisfying $\psi_\alpha = \pi_\alpha \circ \Psi$ for each $\alpha \in \Lambda$.  In this sense, the projective limit $\big (\varprojlim\limits_{\alpha \in \Lambda} \mathcal{A}_\alpha, \{ \pi_\alpha \}_{\alpha \in \Lambda} \big)$ of the projective system $\left ( \{ \mathcal{A}_\alpha \}_{\alpha \in \Lambda},  \{ \pi_{\alpha, \beta} \}_{\alpha \leq \beta} \right )$ is uniquely determined. A reader is directed to \cite{AG, NCP} for more details.

Now observe that for each $\alpha \in \Lambda$, there is a canonical projection map $\phi_\alpha : \mathcal{A} \rightarrow \mathcal{A}_\alpha$  satisfying $\pi_{\alpha, \beta} \circ \phi_\beta = \phi_\alpha$, whenever $\alpha \leq \beta$. Equivalently, we get that the pair $(\mathcal{A}, \{ \phi_\alpha \}_{\alpha \in \Lambda})$ is compatible with the projective system  $\left ( \{ \mathcal{A}_\alpha \}_{\alpha \in \Lambda},  \{ \pi_{\alpha, \beta} \}_{\alpha \leq \beta} \right )$ \big (see Subsection 1.1 of \cite{AG} \big ). Moreover, 
the projective limit $\big(\varprojlim\limits_{\alpha \in \Lambda} \mathcal{A}_\alpha, \{ \phi_\alpha \}_{\alpha \in \Lambda} \big)$ of the projective system $\left ( \{ \mathcal{A}_\alpha \}_{\alpha \in \Lambda},  \{ \pi_{\alpha, \beta} \}_{\alpha \leq \beta} \right )$ can be identified with the locally $C^\ast$-algebra $\mathcal{A}$ and the identification is given by the continuous  bijective $\ast$-homomorphism $\phi : \mathcal{A} \rightarrow \varprojlim\limits_{\alpha \in \Lambda} \mathcal{A}_\alpha$ defined by $a \mapsto \{a + \mathcal{I}_\alpha\}_{\alpha \in \Lambda}$. 
By referring to the identification given by the map $\phi$, we write $\mathcal{A} = \varprojlim\limits_{\alpha \in \Lambda} \mathcal{A}_\alpha$. One may refer to \cite{AG, NCP} and references therein for more details.


\begin{remark} \label{rem;pl,sh} \cite{KS}
Let $\left ( \{ \mathcal{B}_\alpha \}_{\alpha \in \Lambda},  \{ \psi_{\alpha, \beta} \}_{\alpha \leq \beta} \right )$ 
be a projective system of $C^\ast$-algebras. Then the projective limit $\varprojlim\limits_{\alpha \in \Lambda} \mathcal{B}_\alpha$ is a locally $C^\ast$-algebra  and it is uniquely determined by defining as in Equation \eqref{eq;pl} 
\begin{equation*}
\varprojlim\limits_{\alpha \in \Lambda} \mathcal{B}_\alpha := \left \{ \{x_\alpha\}_{\alpha \in \Lambda} \in \prod_{\alpha \in \Lambda} \mathcal{B}_\alpha ~~ : ~~ \psi_{\alpha, \beta}(x_\beta) = x_\alpha, ~~ \text{whenever} ~~ \alpha \leq \beta \right \},
\end{equation*}
along with the continuous $\ast$-homomorphisms $\pi_\beta : \varprojlim\limits_{\alpha \in \Lambda} \mathcal{B}_\alpha \rightarrow \mathcal{B}_\beta$ defined as $\pi_\beta \big  (\{ x_\alpha \}_{\alpha \in \Lambda} \big  ) := x_\beta$ for every $\beta \in \Lambda$. 
\end{remark}

\begin{example} \cite[(1) of Example 1.4]{AG}\label{ex;lca}
Let  $\big \{ \mathcal{H}, \mathcal{F} = \{\mathcal{H}_\alpha\}_{\alpha \in \Lambda}, \mathcal{D} \big \}$ be a quantized domain. For each fixed $\beta \in \Lambda$, define the branch of $\Lambda$ determined by $\beta$ as $\Lambda_\beta = \{ \alpha \in \Lambda \;  :  \; \alpha \leq \beta \}$ . Then, with the induced order from $(\Lambda, \leq)$ the set $(\Lambda_\beta, \leq)$ is a directed poset (see \cite[Section 1.4]{AG}). For each $\beta \in \Lambda$ consider the quantized domain $\big \{ \mathcal{H}_\beta, \mathcal{F}_\beta = \{\mathcal{H}_\alpha\}_{\alpha \in \Lambda_\beta}, \mathcal{H}_\beta \big \}$. Then the $\ast$-algebra $C^*_{\mathcal{F}_\beta}(\mathcal{H}_\beta)$ is a $C^\ast$-subalgebra of $\mathcal{B}(\mathcal{H}_\beta)$ (see \cite[(1) of Example 1.4]{AG}). Now for $\alpha \leq \beta$, define a $\ast$-homomorphism 
\begin{align*}
\phi_{\alpha, \beta} : C^*_{\mathcal{F}_\beta}(\mathcal{H}_\beta) \rightarrow C^*_{\mathcal{F}_\alpha}(\mathcal{H}_\alpha) \; \; \; \text{as} \; \;  \; \phi_{\alpha, \beta}(S) = S \big|_{\mathcal{H}_\alpha},
\end{align*}
for every $S \in C^*_{\mathcal{F}_\beta}(\mathcal{H}_\beta)$. Then $\left ( \big \{ C^*_{\mathcal{F}_\alpha}(\mathcal{H}_\alpha) \big \}_{\alpha \in \Lambda},  \big \{\phi_{\alpha, \beta} \big \}_{\alpha \leq \beta} \right )$ forms a projective system of $C^{\ast}$-algebras. Then by following Equation (1.32) of \cite{AG}, we obtain a locally $C^\ast$-algebra given by
\begin{equation*}
C^*_\mathcal{F}(\mathcal{D}) = \varprojlim\limits_{\alpha \in \Lambda} C^*_{\mathcal{F}_\alpha}(\mathcal{H}_\alpha).
\end{equation*}
Moreover, by referring to Equation (1.33) of \cite{AG}, for each $\alpha \in \Lambda$, one can see that a seminorm defined by $p_\alpha : C^*_\mathcal{F}(\mathcal{D}) \rightarrow \mathbb{R}$ as $p_\alpha(T) := \big \| T\big|_{\mathcal{H}_\alpha} \big \|$ is a $C^\ast$-seminorm. Then the family $\{ p_\alpha \}_{\alpha \in \Lambda}$ of $C^\ast$-seminorms induces a locally convex topology on $C^*_\mathcal{F}(\mathcal{D})$. 
\end{example}

\subsection{Locally von Neumann algebra} In this subsection, we recall some results from the notion of locally von Neumann algebra. We begin with the definition of a locally von Neumann algebra as presented in \cite{MF}.

\begin{definition}\cite[Definition 1.1]{MF} \label{def; lva 1}
An algebra $\mathcal{M}$ is said to be a locally von Neumann algebra if $\mathcal{M}$ is the projective limit of some projective system of von Neumann algebras. 
\end{definition}

Next, we present an alternative approach to defining a locally von Neumann algebra, following the framework developed in \cite{MJ1}. Let  $\big \{ \mathcal{H}; \mathcal{F} = \{\mathcal{H}_\alpha\}_{\alpha \in \Lambda}; \mathcal{D} \big \}$  be a quantized domain. If $u, v \in \mathcal{D}$, then $u, v \in \mathcal{H}_\gamma$ for some $\gamma \in \Lambda$ and we define
\begin{equation*}
q_{u}(T) := \| Tu \|_{\mathcal{H}_\gamma} \; \; \text{and} \; \;q_{u,v}(T) := \left | \langle u, Tv \rangle_{\mathcal{H}_\gamma} \right | \; \; \text{for all} \; T \in C^*_\mathcal{F}(\mathcal{D}).
\end{equation*}
Thus $q_u$ and $q_{u, v}$ are $C^\ast$-seminorms. Then
\begin{enumerate}
\item[(a)] (SOT) {\it strong operator topology} on $C^*_\mathcal{F}(\mathcal{D})$ is the locally convex topology generated by the family $\{q_u ~~ : ~~ u \in \mathcal{D} \}$ of $C^\ast$-seminorms;
\item[(b)] (WOT) {\it weak operator topology} on $C^*_\mathcal{F}(\mathcal{D})$ is the locally convex topology generated by the family $\{q_{u,v} ~~ : ~~ u, v \in \mathcal{D} \}$ of $C^\ast$-seminorms.
\end{enumerate}
For a detailed introduction to locally von Neumann algebras, a reader is directed to \cite{MF, MJ1, MJ2}. In \cite{MJ1}, the author showed that, in the case of $\ast$-subslegrbras of $C^*_\mathcal{F}(\mathcal{D})$), Definition \ref{def; lva 1} coincides with the following definition of locally von Neumann algebras.

\begin{definition}\cite[Definition 3.7]{MJ1} \label{def; lva 2}
Let  $\big \{ \mathcal{H}; \mathcal{F} = \{\mathcal{H}_\alpha\}_{\alpha \in \Lambda}; \mathcal{D} \big \}$  be a quantized domain. Then a locally von Neumann algebra is a strongly closed unital locally $C^{\ast}$-algebra contained in $C^*_\mathcal{F}(\mathcal{D})$ (with the same unit as in $C^*_\mathcal{F}(\mathcal{D})$).
\end{definition}

Let $\mathcal{M} \subseteq C^*_\mathcal{F}(\mathcal{D})$. Consider the set $\mathcal{M}^\prime := \{ T \in C^*_\mathcal{F}(\mathcal{D}) ~~:~~ TS = ST ~~ \text{for all} ~~ S \in \mathcal{M} \}$ which is called as the commutant of $\mathcal{M}$ (see \cite{MJ1}). We denote $(\mathcal{M}^\prime)^\prime$ by the notation $\mathcal{M}^{\prime \prime}$. The following theorem proved in \cite{MJ1} is the double commutant theorem in the setting of locally von Neumann algebra. 

\begin{theorem}\cite[Theorem 3.6]{MJ1}
Let  $\big \{ \mathcal{H}; \mathcal{F} = \{\mathcal{H}_\alpha\}_{\alpha \in \Lambda}; \mathcal{D} \big \}$  be a quantized domain, and let $\mathcal{M} \subseteq C^*_\mathcal{F}(\mathcal{D})$ be a locally $C^{\ast}$-algebra containing the identity operator on $\mathcal{D}$. Then the following statements are equivalent:
\begin{enumerate}
\item $ \mathcal{M} = \mathcal{M}^{\prime \prime}$;
\item $\mathcal{M}$ is weakly closed;
\item $\mathcal{M}$ is strongly closed.
\end{enumerate}
\end{theorem}

\section{Direct integral of locally Hilbert spaces} \label{sec; Direct integrals}
Motivated by the theory of direct integrals of Hilbert spaces, we propose an approach to define the notion of the direct integral of locally Hilbert spaces. To develop this concept, it is necessary to introduce a suitable analogue of a measure space considered in Definition \ref{def;dihs}. 

\subsection{Locally measure space}
To define the notion of locally measure space, we first recall the concept of a strictly inductive system of measurable spaces, as discussed in \cite[Section 4.1]{AG2}.

\begin{definition} \label{def;isms}
Let $\big (\Lambda, \leq \big )$ be a directed poset and let $\big \{(X_\alpha, \Sigma_\alpha) \big \}_{\alpha \in \Lambda}$ be a family of measurable spaces. We say that the family $\big \{(X_\alpha, \Sigma_\alpha) \big \}_{\alpha \in \Lambda}$ forms a strictly inductive system of measurable spaces, if 
\begin{enumerate}
\item $X_\alpha \subseteq X_\beta$;
\item $\Sigma_\alpha = \big\{ E \cap X_\alpha ~~ : ~~ E \in \Sigma_\beta \big \}$ (this implies $\Sigma_\alpha \subseteq \Sigma_\beta$),
\end{enumerate}
whenever $\alpha \leq \beta$.
\end{definition}

\noindent
Next, we recall the construction of the inductive limit in this context as given in \cite[Section 4.1]{AG2}. Suppose $\big \{(X_\alpha, \Sigma_\alpha) \big \}_{\alpha \in \Lambda}$ is a strictly inductive system of measurable spaces, then define 
\begin{equation} \label{eq;sigma algebra}
X := \bigcup\limits_{\alpha \in \Lambda} X_\alpha; \; \; \; \; \Sigma_{0}: =  \bigcup\limits_{\alpha \in \Lambda} \Sigma_\alpha  \; \; \; \; \text{and} \; \; \; \; \; \Sigma := \big \{ E \subseteq X ~~ : ~~ E \cap X_\alpha \in \Sigma_\alpha, ~~ \text{for all} ~~ \alpha \in \Lambda \big \}.
\end{equation}
In Proposition 4.1 of \cite{AG2}, it has been proved that $\Sigma_{0} \subseteq \Sigma$, and the collection $\Sigma$ is a $\sigma$-algebra. In the following example, we illustrate that $\Sigma_0$ is not necessarily a $\sigma$-algebra. 

\noindent

\begin{example} \label{Eg: sigma0}
Consider the family of measurable spaces $\big \{([-n, n], \Sigma_n) \big \}_{n \in \mathbb{N}}$, where $\Sigma_n$ denotes the $\sigma$-algebra of Lebesgue measurable subsets of $[-n, n]$. For each $n \in \mathbb{N}$, we have $[-n, n] \in \Sigma_n$. However, $\bigcup\limits_{n \in \mathbb{N}} [-n, n] = \mathbb{R} \notin \Sigma_0 = \bigcup\limits_{n \in \mathbb{N}} \Sigma_n$. This example demonstrates that $\Sigma_0$ is not necessarily a $\sigma$-algebra. 
\end{example}

\noindent
Next, we introduce the notion of projective system of measures. 

\begin{definition} \label{def;sisms}
Let $\big (\Lambda, \leq \big )$ be a directed poset and let $\big \{(X_\alpha, \Sigma_\alpha) \big \}_{\alpha \in \Lambda}$ be a strictly inductive system of measurable spaces. Suppose $\mu_\alpha$ is a positive measure on the measurable space $(X_\alpha, \Sigma_\alpha)$ for each $\alpha \in \Lambda$. Then the family $\{\mu_\alpha\}_{\alpha \in \Lambda}$ is said to be a \textit{projective system of measures}, if for each $E_\alpha \in \Sigma_\alpha$, we have
\begin{equation*}
\mu_\alpha(E_\alpha) = \mu_\beta(E_\alpha),  \; \; \text{whenever} \; \;  \alpha \leq \beta.
\end{equation*}
This implies that for every $E \in \Sigma$, we see that 
\begin{equation*}
\mu_\alpha(E \cap X_\alpha) = \mu_\beta(E \cap X_\alpha)
\leq \mu_\beta(E \cap X_\beta).
\end{equation*}
\end{definition}

\begin{proposition} \label{prop;m}
Suppose $\big \{(X_\alpha, \Sigma_\alpha, \mu_\alpha) \big \}_{\alpha \in \Lambda}$ is a family of measure spaces such that $\big \{(X_\alpha, \Sigma_\alpha) \big \}_{\alpha \in \Lambda}$ is a strictly inductive system of measurable spaces and $\{\mu_\alpha\}_{\alpha \in \Lambda}$ is a projective system of measures. If $X$ and $\Sigma$ be as defined in Equation \eqref{eq;sigma algebra},  then the map $\mu : \Sigma \rightarrow [0, \infty]$ defined by 
\begin{equation*}
\mu(E) := \begin{cases}
\lim\limits_\alpha \; \mu_\alpha(E \cap X_\alpha), & \text{if} \;\; \{ \mu_\alpha(E \cap X_\alpha) \}_{\alpha \in \Lambda} \; \text{converges};\\
& \\
\infty, & \text{otherwise}
\end{cases} 
\end{equation*} 
is a measure on $(X, \Sigma)$. 
\end{proposition}
\begin{proof}
Let $\emptyset$ denote the empty set of $X$. Then $\mu(\emptyset) = \lim\limits_\alpha \{ \mu_\alpha(\emptyset \cap X_\alpha) \} = 0.$ Further, assume that $\{ E_n \in \Sigma ~~ : ~~  n \in \mathbb{N} \}$ is a collection of pairwise disjoint subsets of $X$ in $\Sigma$. Then, we obtain
\begin{align*} 
\mu \big (\bigcup_{n \in \mathbb{N}} E_n \big ) &= \lim\limits_{\alpha} \; \mu_\alpha \big (\big (\bigcup_{n \in \mathbb{N}} E_n \big ) \bigcap X_\alpha \big ) \\
&= \lim\limits_{\alpha} \; \mu_\alpha  \big (\bigcup_{n \in \mathbb{N}} \left (E_n \bigcap X_\alpha \right ) \big) \\
&= \lim\limits_{\alpha} \; \sum_{n=1}^{\infty} \mu_\alpha \big (E_n \bigcap X_\alpha \big ) 
\end{align*}
and 
\begin{align*} 
\sum_{n=1}^{\infty} \mu(E_n) = \sum_{n=1}^{\infty} \lim\limits_{\alpha} \; \mu_\alpha \big  (E_n \bigcap X_\alpha \big).
\end{align*}

For each $\alpha \in \Lambda$, define a function $f_\alpha : \mathbb{N} \rightarrow [0, \infty)$ by $f_\alpha(n) := \mu_\alpha(E_n \cap X_\alpha).$ Whenever $\alpha \leq \beta$, we have $f_\alpha(n) = \mu_\alpha(E_n \cap X_\alpha) \leq f_\beta(n) = \mu_\beta(E_n \cap X_\beta)$ for all $n \in \mathbb{N}$. If we define a function $f : \mathbb{N} \rightarrow [0, \infty]$ by $f(n) := \mu(E_n)$, then for all $n \in \mathbb{N}$ we get $\lim\limits_\alpha f_\alpha(n) = f(n)$. By using the monotone convergence theorem, we have 
\begin{equation*}
\lim\limits_\alpha \sum\limits_{n=1}^{\infty} f_\alpha(n) = \sum\limits_{n=1}^{\infty} f(n).
\end{equation*}
This implies that
\begin{align*}
\lim\limits_\alpha \sum_{n=1}^{\infty} \mu_\alpha \left (E_n \bigcap X_\alpha \right ) = \lim\limits_\alpha \sum\limits_{n=1}^{\infty} f_\alpha(n) = \sum\limits_{n=1}^{\infty} f(n) &= \sum_{n=1}^{\infty} \mu (E_n) \\
&= \sum_{n=1}^{\infty} \lim\limits_\alpha \mu_\alpha \big (E_n \bigcap X_\alpha \big ) \\
&= \mu \big (\bigcup_{n \in \mathbb{N}} E_n \big ).
\end{align*}
This proves $\mu \big (\bigcup\limits_{n \in \mathbb{N}} E_n \big ) = \sum\limits_{n=1}^{\infty} \mu(E_n)$, and hence the map $\mu$ is a measure on $(X, \Sigma)$.
\end{proof}

\noindent
Now, we introduce the notion of locally measure space. 

\begin{definition} \label{def; lms}
We call the measure space $(X, \Sigma, \mu)$ obtained in Proposition \ref{prop;m} a locally measure space. 
\end{definition}

\begin{example} \label{ex;lsms}
Let $\Lambda = \mathbb{N}$ and $X_n = [-n, n]$ for each $n \in \mathbb{N}$. Suppose $B(X_n)$ denotes the Borel $\sigma$-algebra of $X_n$ and $\mu_n$ denotes the Lebesgue measure on $B(X_n)$. Then $\{(X_n, B(X_n), \mu_n)\}_{n \in \mathbb{N}}$ is a family of measure spaces such that $\{(X_n, B(X_n)\}_{n \in \mathbb{N}}$ is a strictly inductive system of measurable spaces and $\{\mu_n\}_{n \in \mathbb{N}}$ is a projective system of measures. As we know that a subset $U$ of $\mathbb{R}$ is Borel if and only if $U \cap X_n$ is Borel for every $n \in \mathbb{N}$. Moreover, if $\mu$ is the Lebesgue measure of $\mathbb{R}$ and $E \in B(\mathbb{R})$, then $E = \bigcup\limits_{n \in \mathbb{N}}(E \cap X_n)$, where 
\begin{equation*}
E \cap X_m \subseteq E \cap X_n, \; \; \text{whenever} \; \; m \leq n \; \; \text{and} \; \; \mu_n(E \cap X_n) = \mu (E \cap X_n) \; \; \text{for all} \; \; n \in \mathbb{N}.
\end{equation*} 
This implies that either $\mu(E) = \infty$ or $\mu(E) = \lim\limits_{n \to \infty} \mu (E \cap X_n) = \lim\limits_{n \to \infty} \mu_n (E \cap X_n)$. This shows that $(\mathbb{R}, B(\mathbb{R}), \mu)$ is a locally measure space.
\end{example}

\begin{note} \label{note; lms}
In the remaining part of this article, $(\Lambda, \leq)$ will denote a directed poset and the notation $(X, \Sigma, \mu)$ will always indicate a locally measure space associated with the family $\{ (X_\alpha, \Sigma_\alpha, \mu_\alpha)\}_{\alpha \in \Lambda}$ of measure spaces, where $X_\alpha$ is a $\sigma$-compact locally compact space, $\Sigma_\alpha$ is a Borel $\sigma$-algebra on $X_\alpha$ and  $\mu_\alpha$ is the completion of a positive Borel measure on $X_\alpha$ (see Definition \ref{def;dihs}), unless otherwise stated.
\end{note}


\subsection{Direct integral of locally Hilbert spaces}
Now we are in a position to propose the notion of direct integral of locally Hilbert spaces over a locally measure space.

\begin{definition} \label{Defn: directint_loc}
Let $(X, \Sigma, \mu)$ be a locally measure space (see Note \ref{note; lms}). For each $p \in X$,  assign a quantized domain $\big \{ \mathcal{H}_p; \mathcal{E}_p = \{\mathcal{H}_{\alpha, p}\}_{\alpha \in \Lambda}; \mathcal{D}_p \big \}$. Then the direct integral of locally Hilbert spaces $\{ \mathcal{D}_p \}_{p \in X}$ is, by definition, a space $\mathcal{D}$ collection of all maps $u : X \rightarrow \bigcup\limits_{p \in X} \mathcal{D}_p$ such that $u(p) \in \mathcal{D}_p$ for all $p \in X$ satisfying the following conditions: 
\begin{enumerate}
\item for each $u \in \mathcal{D}$ there exists $\alpha_u \in \Lambda$ such that $u(p) \in \mathcal{H}_{\alpha_u, p}$ for a.e. $p \in X_{\alpha_u}$  and 
\begin{equation*}
\text{supp}(u) := \{ p \in X \; : \; u(p) \neq 0_{\mathcal{D}_p}  \} \subseteq X_{\alpha_u};
\end{equation*}    
\item for any $u, v \in \mathcal{D}$, the function $\zeta_{u,v} : X \rightarrow \mathbb{C}$ defined by 
\begin{equation*}
\zeta_{u,v}(p) := \big \langle u(p), v(p) \big \rangle_{\mathcal{D}_p}, \; \; \text{for all}\; p \in X
\end{equation*}
is in  $\text{L}^1(X, \mu)$;
\item if $v : X \rightarrow \bigcup\limits_{p \in X} \mathcal{D}_p$ is such that there exists $\alpha \in \Lambda$ with  $v(p) \in \mathcal{H}_{\alpha, p}$ for all $p \in X_\alpha$ and $v(p) = 0_{\mathcal{D}_p}$ for all $p \in X \setminus X_{\alpha}$, and, in addition, for each $u \in \mathcal{D}$ the function $\zeta_{u,v} \in \text{L}^1(X, \mu)$, then $v \in \mathcal{D}.$  
\end{enumerate}
\end{definition}

\begin{notation} \label{noatation; direct integral}
We denote the collection $\mathcal{D}$ by $\displaystyle \dilX \mathcal{D}_p \, \dmu$ and read it as the direct integral of locally Hilbert spaces $\{ \mathcal{D}_p \}_{p \in X}$ over the locally measure space $(X, \Sigma, \mu)$. Here $``\oplus_{\text{loc}}"$ represents that the underlying spaces are locally Hilbert space, whereas $``X_{\text{loc}}"$ indicates that the measure space $(X, \Sigma, \mu)$ is a locally measure space as described in Note \ref{note; lms}. We denote an element $u \in \displaystyle \dilX \mathcal{D}_p \, \dmu$ by $\displaystyle \dilX u(p) \, \dmu$. In particular, in Note \ref{note; lms}, if each $X_\alpha = X$, then we denote $\mathcal{D}$ by $\displaystyle \int^{\oplus_\text{loc}}_{X} \mathcal{D}_p \, \dmu$ and $u \in \displaystyle \int^{\oplus_\text{loc}}_{X} \mathcal{D}_p \, \dmu$ by $\displaystyle \int^{\oplus_\text{loc}}_{X} u(p) \, \dmu$.    
\end{notation}

\begin{note} \label{note; V alpha star}
Consider a locally measure space $(X, \Sigma, \mu)$ and a family  $\big \{ \mathcal{H}_p; \mathcal{E}_p = \{\mathcal{H}_{\alpha, p}\}_{\alpha \in \Lambda}; \mathcal{D}_p \big \}_{p \in X}$ of quantized domains. Corresponding to any fixed $\alpha \in \Lambda$ and $x \in \int^\oplus_{X_\alpha} \mathcal{H}_{\alpha, p} \, \mathrm{d} \mu_\alpha (p)$, define a map $u_x : X \rightarrow \bigcup\limits_{p \in X} \mathcal{D}_p$ by
\begin{equation}\label{eq; ux defined with x}
u_x(p) := \begin{cases}
x(p), & \text{if} \;\; p \in X_\alpha;\\
0_{\mathcal{D}_p} & p \in X \setminus X_\alpha.
\end{cases}
\end{equation}
Then $u_x \in \displaystyle \dilX \mathcal{D}_p \, \dmu$. 
\end{note}


We use the construction of an element $u_x$ in $\displaystyle \dilX \mathcal{D}_p \, \dmu$ corresponding to a (arbitrarily) fixed $\alpha \in \Lambda$ and $x \in \int^\oplus_{X_\alpha} \mathcal{H}_{\alpha, p} \, \mathrm{d} \mu_\alpha (p)$  described in Note \ref{note; V alpha star} to show that the set $\displaystyle\dilX \mathcal{D}_p \, \dmu$ is indeed a locally Hilbert space.

\begin{proposition} \label{prop;dilhs}
Let $(X, \Sigma, \mu)$ be a locally measure space and $\big \{ \mathcal{H}_p; \mathcal{E}_p = \{\mathcal{H}_{\alpha, p}\}_{\alpha \in \Lambda}; \mathcal{D}_p \big \}_{p \in X}$ be a family of quantized domains. Then the set $\displaystyle \dilX \mathcal{D}_p \, \dmu$ of direct integral of locally Hilbert spaces $\{ \mathcal{D}_p \}_{p \in X}$ is a locally Hilbert space. 
\end{proposition}
\begin{proof}
To prove that $\displaystyle \dilX  \mathcal{D}_p \, \dmu$ is a locally Hilbert space, first we construct a strictly inductive system of Hilbert spaces. For each fixed $\alpha \in \Lambda$, we define the set $\mathcal{H}_\alpha$ by
\begin{equation}\label{eq; H alpha}
\mathcal{H}_\alpha := \left \{ u \in \dilX \mathcal{D}_p \, \dmu ~~ : ~~  u(p) \in \mathcal{H}_{\alpha, p} \; \; \text{for a.e.} \; \; p \in X_\alpha  \; \; \; \text{and} \; \;  \text{supp}(u) \subseteq X_\alpha \right \}.
\end{equation}
The set $\mathcal{H}_\alpha$ can be $\{ 0_{\mathcal{H}_\alpha} \}$ for some $\alpha \in \Lambda$. For instance, if the family $\{ \mathcal{D}_p \}_{p \in X}$ is such that for a fixed $\alpha \in \Lambda$, $\mathcal{H}_{\alpha, p} = \{ 0_{\mathcal{H}_{\alpha, p}}  \}$ for a.e. $p \in X_\alpha$, then we get that $\mathcal{H}_\alpha$ to be the zero space. On the other hand, let $\alpha \in \Lambda$ be fixed, and $E \subseteq X_\alpha$ be such that $0 < \mu (E) = \mu_\alpha(E) < \infty$ with $\mathcal{H}_{\alpha, p}$ is non-trivial Hilbert space for all $p \in E$. Then consider the family $\{ v_p \}_{p \in X}$, where $v_p$ is a unit vector in $\mathcal{H}_{\alpha, p} \subseteq \mathcal{D}_p$ if $p \in E$ and $v_p = 0_{\mathcal{D}_p}$ if $p \in X \setminus E$. Suppose $u : X \rightarrow \bigcup\limits_{p \in X} \mathcal{D}_p$ is a function satisfying the property (1) of Definition \ref{Defn: directint_loc} and $\zeta_{u,u} \in \text{L}^1(X, \mu)$, \big(equivalently, the map $p \mapsto \big \| 
u(p) \big \|_{\mathcal{H}_{{\alpha_u}, p}}$ is in $\text{L}^2(X, \mu)$ \big) then 
\begin{align*}
\int_{X} \; \big | \la u(p), v_p  \ra \big| \; \dmu &\leq \int\limits_{X_{\alpha_u} \bigcap E} \; \| u(p) \|_{\mathcal{H}_{{\alpha_u}, p}}\; \| v_p \|_{\mathcal{H}_{\alpha, p}}\; \dmu \\
&= \int\limits_{X_{\alpha_u} \bigcap E} \; \| u(p) \|_{\mathcal{H}_{{\alpha_u}, p}}  \; \mathrm{d} \mu_{\alpha_u} < \infty.
\end{align*}
The last inequality holds true as the map $p \mapsto \big \| 
u(p) \big \|_{\mathcal{H}_{{\alpha_u}, p}}$ when restricted to $X_{\alpha_u} \cap E$ is in $\text{L}^1(X_{\alpha_u}, \mu_{\alpha_u})$, as one may see that $\mu_{\alpha_u}(X_{\alpha_u} \cap E) < \infty$. So, from the property (3) of Definition \ref{Defn: directint_loc}, there exists $v \in \displaystyle \dilX  \mathcal{D}_p \, \dmu$ such that $v(p) = v_p$ if $p \in X_\alpha$ and $v(p) = 0_{\mathcal{D}_p}$ if $p \in X \setminus X_\alpha$. In particular, $0_{\mathcal{H}_\alpha} \neq v \in \mathcal{H}_\alpha$. Thus $\mathcal{H}_\alpha$ is non-trivial.

Let $u, v \in \mathcal{H}_\alpha$. Then define
\begin{equation} \label{eq;ip alpha}
\langle u, v \rangle_{\mathcal{H}_\alpha} : = \int\limits_{X_{\alpha}} \big \langle u(p),\; v(p) \big \rangle_{\mathcal{H}_{\alpha, p}} \; \mathrm{d} \mu_{\alpha}.
\end{equation} 
This gives an inner product on $\mathcal{H}_\alpha$. Whenever $\alpha \leq \beta$, it is clear that $\mathcal{H}_\alpha \subseteq \mathcal{H}_\beta$ and since $X_\alpha \subseteq X_\beta$ (see Definition \ref{def;isms}), for $u, v \in \mathcal{H}_\alpha$, we have
\begin{equation} \label{eq;ip alpha, beta}
\langle u, v \rangle_{\mathcal{H}_\beta} = \int\limits_{X_{\beta}} \big \langle u(p),\; v(p) \big \rangle_{\mathcal{H}_{\beta, p}} \mathrm{d} \mu_{\beta} 
= \int\limits_{X_{\alpha}} \big \langle u(p),\; v(p) \big \rangle_{\mathcal{H}_{\alpha, p}} \mathrm{d} \mu_{\alpha} + \int\limits_{X_\beta \setminus X_{\alpha}} \big \langle u(p),\; v(p) \big \rangle_{\mathcal{H}_{\beta, p}} \mathrm{d} \mu_{\alpha} 
= \langle u, v \rangle_{\mathcal{H}_\alpha}.
\end{equation}
Thus the inclusion map $J_{\beta, \alpha} : \mathcal{H}_\alpha \rightarrow \mathcal{H}_\beta$ is an isometry. Now for each $\alpha \in \Lambda$, we will show that the inner product space $\mathcal{H}_\alpha$ is complete by identifying it with the Hilbert space $\int^\oplus_{X_\alpha} \mathcal{H}_{\alpha, p} \, \mathrm{d} \mu_\alpha (p)$, which is the direct integral of the family $\{ \mathcal{H}_{\alpha, p} \}_{p \in X_\alpha}$ of Hilbert spaces over the measure space $\big (X_\alpha, \Sigma_\alpha, \mu_\alpha \big )$. To see this, define an operator $V_\alpha : \mathcal{H}_\alpha \rightarrow \int^\oplus_{X_\alpha} \mathcal{H}_{\alpha, p} \, \mathrm{d} \mu_\alpha (p)$ by 
\begin{equation} \label{eq;iso}
V_\alpha(u)(p) := u(p) \; \; \text{for all} \; \; p \in X_\alpha \; \; \text{and} \; \; u \in \mathcal{H}_\alpha.
\end{equation}
Then $V_\alpha$ is an isometry, that is,
\begin{equation*}
\big \| u \big \|_{\mathcal{H}_\alpha}^2 =
\int\limits_{X_{\alpha}} \big \langle u(p),\; u(p) \big \rangle_{\mathcal{H}_{\alpha, p}}\; \mathrm{d} \mu_{\alpha} = \int\limits_{X_{\alpha}} \big \langle V_\alpha(u)(p),\; V_\alpha(u)(p)\big \rangle_{\mathcal{H}_{\alpha, p}} \; \mathrm{d} \mu_{\alpha}  = \big \| V_\alpha(u) \big \|^2,
\end{equation*}
and for every $x \in \int^\oplus_{X_\alpha} \mathcal{H}_{\alpha, p} \, \mathrm{d} \mu_\alpha (p)$ we have $V_\alpha(u_x) = x$, where $u_x$ is defined as in Equation \eqref{eq; ux defined with x}. This shows that $V_\alpha$ is surjective. Hence, the map $V_\alpha$ defines an isomorphism between the inner product space $\mathcal{H}_\alpha$ and the Hilbert space $\int^\oplus_{X_\alpha} \mathcal{H}_{\alpha, p} \, \mathrm{d} \mu_\alpha (p)$. Since $\alpha \in \Lambda$ is arbitrary, each $\mathcal{H}_\alpha$ is a Hilbert space. Thus, $\{ \mathcal{H}_\alpha \}_{\alpha \in \Lambda}$ forms a strictly inductive system of Hilbert spaces and from Equation \eqref{eq;lhs}, we get that 
\begin{equation} \label{eq;il=u}
\varinjlim\limits_{\alpha \in \Lambda} \mathcal{H}_\alpha = \bigcup\limits_{\alpha \in \Lambda} \mathcal{H}_\alpha.
\end{equation} 

Finally, to prove the result, we show that $\displaystyle \dilX \mathcal{D}_p \, \dmu$ is the inductive limit of the strictly inductive system $\{ \mathcal{H}_\alpha \}_{\alpha \in \Lambda}$ of Hilbert spaces. Clearly, $\bigcup\limits_{\alpha \in \Lambda} \mathcal{H}_\alpha \subseteq \displaystyle \dilX \mathcal{D}_p \, \dmu$ by Equation \eqref{eq; H alpha}. If $u \in \displaystyle \dilX \mathcal{D}_p \, \dmu$, then from the property (1) of Definition \ref{Defn: directint_loc} there exists $\alpha_u \in \Lambda$ such that  $u(p) \in \mathcal{H}_{\alpha_u, p}$ for a.e. $p \in X_{\alpha_u}$ and $\text{supp}(u) \subseteq X_{\alpha_u}$. This implies that $u \in \mathcal{H}_{\alpha_{u}}$. Hence, we get $\displaystyle \dilX \mathcal{D}_p \, \dmu \subseteq \bigcup\limits_{\alpha \in \Lambda} \mathcal{H}_\alpha$, and this proves 
\begin{equation} \label{eq;dilhs=u}
\dilX \mathcal{D}_p \, \dmu = \bigcup\limits_{\alpha \in \Lambda} \mathcal{H}_\alpha.
\end{equation}
Then Equation \eqref{eq;il=u} and Equation \eqref{eq;dilhs=u} imply that 
\begin{equation} \label{eq;dilhs=il}
\dilX \mathcal{D}_p \, \dmu = \varinjlim\limits_{\alpha \in \Lambda} \mathcal{H}_\alpha.
\end{equation} 
Therefore, from Definition \ref{def;lhs}, we conclude that $\displaystyle \dilX \mathcal{D}_p \, \dmu$ is a locally Hilbert space.
\end{proof}


\begin{remark} \label{rem; quntized domain}
Let $(X, \Sigma, \mu)$ be a locally measure space and $\big \{ \mathcal{H}_p; \mathcal{E}_p = \{\mathcal{H}_{\alpha, p}\}_{\alpha \in \Lambda}; \mathcal{D}_p \big \}_{p \in X}$ be a family of quantized domains. By following Proposition \ref{prop;dilhs}, we know that the set $\displaystyle \dilX \mathcal{D}_p \, \dmu$  forms a locally Hilbert space. Let $\overline{ \displaystyle \dilX \mathcal{D}_p \, \dmu}$ be the Hilbert space completion of the locally Hilbert space $\displaystyle \dilX \mathcal{D}_p \, \dmu$. Then we obtain 
\begin{equation*}
\left \{ \overline{\dilX \mathcal{D}_p \, \dmu}; \mathcal{E} = \{\mathcal{H}_{\alpha}\}_{\alpha \in \Lambda}; \dilX \mathcal{D}_p \, \dmu \right \},
\end{equation*}
as a quantized domain, where for each $\alpha \in \Lambda$, the Hilbert space $\mathcal{H}_{\alpha}$ is defined as in Equation \eqref{eq; H alpha}. 
\end{remark}

Now, we present some examples of direct integral of locally Hilbert spaces.

\begin{example} \label{ex;dilhs direct sum and direct integral}
Consider $\big (\Lambda = \mathbb{N}, \leq \big )$ a directed poset. For each $i \in \Lambda$, consider a measure space $\big \{ \big ( X_i = \{1, 2, 3, ..., i \}, \Sigma_i, \mu_i \big ) \big \}$. Following Equation \eqref{eq;sigma algebra} and Definition \ref{def; lms}, we get a locally measure space $\big ( X = \mathbb{N}, \Sigma, \mu \big )$, where $\mu$ denotes the counting measure on $\mathbb{N}$. Let $\big \{ \mathcal{H}_n; \mathcal{E}_n = \{\mathcal{H}_{i, n}\}_{i \in \mathbb{N}}; \mathcal{D}_n \big \}_{n \in \mathbb{N}}$ be a family of quantized domains. Extending the notion of the finite direct sum of locally Hilbert spaces to the infinite case, we obtain
\begin{equation} \label{eq; direct sum infinite case}
\bigoplus\limits_{n =1}^\infty \mathcal{D}_n := \big \{ x = (x_1, x_2, ...) \; : \; x_n \in \mathcal{D}_n  \; \; \text{for all} \; n \in \mathbb{N} \; \; \text{and} \; \; \text{supp}(x) < \infty \; \big \} = \bigcup\limits_{i \in \mathbb{N}}  \left (\bigoplus\limits_{n =1}^\infty \mathcal{H}_{i, n} \right ).
\end{equation}
Next, we consider the direct integral of the family $\{ \mathcal{D}_n \}_{n \in \mathbb{N}}$ of locally Hilbert spaces and establish the following equality 
\begin{equation*}
\int^{\oplus_\text{loc}}_{\mathbb{N}_{\text{loc}}} \mathcal{D}_n \; \mathrm{d} \mu(n) = \bigoplus\limits_{n =1}^\infty \mathcal{D}_n.
\end{equation*}
To prove the equality, we begin by observing that $\displaystyle\int^{\oplus_\text{loc}}_{\mathbb{N}_{\text{loc}}} \mathcal{D}_n \; \mathrm{d} \mu(n) = \bigcup_{i \in \mathbb{N}} \mathcal{H}_i$, where for each $i \in \mathbb{N}$, the Hilbert space $\mathcal{H}_i$ is defined as (refer Equation \eqref{eq; H alpha})
\begin{equation*}
\mathcal{H}_i := \left \{ u \in \int^{\oplus_\text{loc}}_{\mathbb{N}_{\text{loc}}} \mathcal{D}_n \; \mathrm{d} \mu(n) ~~ : ~~  u(n) \in \mathcal{H}_{i, n} \; \; \text{for a.e.} \; \; n \in X_i  \; \; \; \text{and} \; \;  \text{supp}(u) \subseteq X_i \right \}.
\end{equation*}
Suppose $u \in \displaystyle\int^{\oplus_\text{loc}}_{\mathbb{N}_{\text{loc}}} \mathcal{D}_n \; \mathrm{d} \mu(n)$. Then $u \in \mathcal{H}_i$ for some $i \in \mathbb{N}$. Thus, $u : \mathbb{N} \rightarrow \bigcup\limits_{n \in X} \mathcal{D}_n$ with $u(n) \in \mathcal{D}_n$ for all $n \in X$. Moreover, $u(n) \in \mathcal{H}_{i, n}$ for each $n \in X_i = \{1, 2, 3, ..., i \}$ and $\text{supp}(u) := \{ n \in \mathbb{N} \; : \; u(n) \neq 0_{\mathcal{D}_n}  \} \subseteq \{1, 2, 3, ..., i \}$. Therefore, $u$ belongs to the Hilbert space $\bigoplus\limits_{n =1}^\infty \mathcal{H}_{i, n}$, and hence $u \in \bigoplus\limits_{n =1}^\infty \mathcal{D}_n$. This shows that $\displaystyle \int^{\oplus_\text{loc}}_{\mathbb{N}_{\text{loc}}} \mathcal{D}_n \; \mathrm{d} \mu(n) \subseteq \bigoplus\limits_{n =1}^\infty \mathcal{D}_n$. 

Conversely, suppose $u \in \bigoplus\limits_{n =1}^\infty \mathcal{D}_n$. Then, by Equation \eqref{eq; direct sum infinite case}, there exists some $i \in \mathbb{N}$ such that $u \in \bigoplus\limits_{n =1}^\infty \mathcal{H}_{i, n}$. Hence, there exists $k \in \mathbb{N}$ such that the support of $u$ is contained in $\{1, 2, 3, ..., k \}$ and we obtain 
\begin{equation*}
u = \big (u(1), u(2), ..., u(k), 0_{\mathcal{D}_{k+1}}, 0_{\mathcal{D}_{k+2}}, ... \big ).
\end{equation*}
Let $r := \text{max}\{ i, k \}$. Then for every $n \in \{1, 2, 3, ..., r \}$, we have $u(n) \in \mathcal{H}_{r, n}$, and for all $n > r$, $u(n) = 0_{\mathcal{D}_{n}}$. Thus $u$ satisfies condition (1) from Definition \ref{Defn: directint_loc} and as the support of $u$ is finite it follows that $u$ satisfies the conditions (2) and (3) from Definition \ref{Defn: directint_loc}. In particular,  $u \in \mathcal{H}_r \subseteq \displaystyle\int^{\oplus_\text{loc}}_{\mathbb{N}_{\text{loc}}} \mathcal{D}_n \; \mathrm{d} \mu(n)$ and we conclude that $\displaystyle \bigoplus\limits_{n =1}^\infty \mathcal{D}_n \subseteq  \int^{\oplus_\text{loc}}_{\mathbb{N}_{\text{loc}}} \mathcal{D}_n \; \mathrm{d} \mu(n)$. This proves the desired equality $\displaystyle\int^{\oplus_\text{loc}}_{\mathbb{N}_{\text{loc}}} \mathcal{D}_n \; \mathrm{d} \mu(n) = \bigoplus\limits_{n =1}^\infty \mathcal{D}_n$. 
\end{example}

\begin{remark}
In Example \ref{ex;dilhs direct sum and direct integral}, we established that $\displaystyle\int^{\oplus_\text{loc}}_{\mathbb{N}_{\text{loc}}} \mathcal{D}_n \; \mathrm{d} \mu(n) = \bigoplus\limits_{n =1}^\infty  \mathcal{D}_n$. However, for a fixed $i \in \Lambda$ the Hilbert space $\mathcal{H}_i$ is not equal to $\bigoplus\limits_{n =1}^\infty \mathcal{H}_{i, n}$. The reason is that in $\mathcal{H}_i$ vectors are supported on the subset of a finite set $X_i = \{1, 2, 3, ..., i \}$. In contrast, the Hilbert space $\bigoplus\limits_{n =1}^\infty \mathcal{H}_{i, n}$ consists of sequences with finite support, but not necessarily within $\{1, 2, 3, ..., i \}$. However, to obtain the equality between the Hilbert spaces $\mathcal{H}_i$ and $\bigoplus\limits_{n =1}^\infty \mathcal{H}_{i, n}$, one must replace the measure space $\big \{ \big ( X_i = \{1, 2, 3, ..., i \}, \Sigma_i, \mu_i \big ) \big \}$ with $\big ( X = \mathbb{N}, \Sigma, \mu \big )$ for each $i \in \Lambda$. That is, one has to consider $\displaystyle\int^{\oplus_\text{loc}}_{\mathbb{N}} \mathcal{D}_n \; \mathrm{d} \mu(n)$ instead of $\displaystyle\int^{\oplus_\text{loc}}_{\mathbb{N}_{\text{loc}}} \mathcal{D}_n \; \mathrm{d} \mu(n)$ (see Notation \ref{noatation; direct integral}).
\end{remark}

\begin{example} \label{ex;dilhs direct integral of C}
Consider a measure space $\big ( \mathbb{R}, B(\mathbb{R}), \mu \big)$ and a family $\{ \mathcal{H}_p = \mathbb{C} \}_{p \in \mathbb{R}}$ of Hilbert spaces.  Clearly, we know that the direct integral of the family $\{ \mathcal{H}_p = \mathbb{C} \}_{p \in \mathbb{R}}$ of Hilbert spaces over the measure space $\big ( \mathbb{R}, B(\mathbb{R}), \mu \big)$ is given by 
\begin{equation*}
\int^\oplus_{\mathbb{R}} \mathcal{H}_p \, \mathrm{d} \mu(p) = \int^\oplus_{\mathbb{R}} \mathbb{C} \, \mathrm{d} \mu(p) = \text{L}^2 \big (\mathbb{R}, \mu \big).
\end{equation*}
Now consider a directed poset $\big (\Lambda = \mathbb{N}, \leq \big )$. For each $n \in \mathbb{N}$, consider a measure space  $\big \{(X_n = [-n, n], B(X_n), \mu_n) \big \}$. By following the procedure given in Example \ref{ex;lsms}, we obtain a locally measure space $\big ( X = \mathbb{R}, B(\mathbb{R}), \mu \big)$. For each $p \in \mathbb{R}$, let $\left \{ \mathcal{H}_p = \mathbb{C}; \mathcal{E}_p = \{ \mathcal{H}_{n,p} = \mathbb{C} \}_{n \in \mathbb{N}}; \mathcal{D}_p = \mathbb{C} \right \}$ be a  quantized domain. Then by Proposition \ref{prop;dilhs}, we obtain 
\begin{equation*}
\displaystyle \int^{\oplus_\text{loc}}_{\mathbb{R}_\text{loc}} \mathcal{D}_p \, \dmu = \displaystyle \int^{\oplus_\text{loc}}_{\mathbb{R}_\text{loc}} \mathbb{C} \, \dmu =  \bigcup\limits_{n \in \mathbb{N}} \mathcal{H}_n,
\end{equation*}
where for each $n \in \mathbb{N}$,
\begin{equation} \label{eq; Hn in example}
\mathcal{H}_n := \left \{ u \in \int^{\oplus_\text{loc}}_{\mathbb{R}_\text{loc}} \mathbb{C} \, \dmu \; \; : \; \; u(p) \in \mathcal{H}_{n, p} = \mathbb{C} \; \text{for a.e.} \;   p \in X_n = [-n,n] \; \; \text{and} \; \; \text{supp}(u) \subseteq [-n,n] \right \},
\end{equation}
which is isomorphic to $\text{L}^2 \big ([-n,n], \mu_n \big)$. Thus, $\mathcal{H}_n$ consists of all Borel measurable functions $u : \mathbb{R} \rightarrow \mathbb{C}$ such that $\text{supp}(u) \subseteq [-n,n]$ with $\int\limits_{\mathbb{R}} |u(p)|^{2} \; \mathrm{d} \mu < \infty$. Therefore, from Equation \eqref{eq;dilhs=u} we obtain 
\begin{equation} \label{eq; direct integral of C}
\int^{\oplus_\text{loc}}_{\mathbb{R}_\text{loc}} \mathbb{C} \, \dmu 
 = \Big\{ u \in \text{L}^2 \big (\mathbb{R}, \mu \big) \; \; : \; \; \text{supp}(u) \subseteq [-n, n] \; \; \text{for some} \; \; n \in \mathbb{N}  \Big\},
\end{equation}
which is the collection of all functions in $\text{L}^2 \big (\mathbb{R}, \mu \big)$ with compact support and hence the locally Hilbert space $\displaystyle \int^{\oplus_\text{loc}}_{\mathbb{R}_\text{loc}} \mathbb{C} \, \dmu$ is dense in the Hilbert space $\int^{\oplus}_{\mathbb{R}} \mathbb{C} \, \dmu$. 
\end{example}

\begin{example} \label{ex;dilhs direct integral of L2 R mu}
Consider a measure space $\big ( \mathbb{R}, B(\mathbb{R}), \mu \big)$ and a family $\{ \mathcal{H}_p = \text{L}^{2}(\mathbb{R}, \mu) \}_{p \in \mathbb{R}}$ of Hilbert spaces. We know that the direct integral of the family $\big \{ \mathcal{H}_p = \text{L}^{2}(\mathbb{R}, \mu) \big \}_{p \in \mathbb{R}}$ of Hilbert spaces over the measure space $\big ( \mathbb{R}, B(\mathbb{R}), \mu \big)$ is given by 
\begin{equation*}
\int^\oplus_{\mathbb{R}} \mathcal{H}_p \, \mathrm{d} \mu(p) = \int^\oplus_{\mathbb{R}} \text{L}^{2}(\mathbb{R}, \mu) \, \mathrm{d} \mu(p) \cong \text{L}^2 \big (\mathbb{R}, \mu \big) \otimes \text{L}^{2}(\mathbb{R}, \mu).
\end{equation*}
Now consider a directed poset $\big ( \Lambda = [0, \infty), \leq \big )$. For each $\alpha \in [0, \infty)$ consider a measure space $\big \{ \big ( X_\alpha = [-\alpha, \alpha], B \big ([-\alpha, \alpha] \big), \mu_\alpha \big ) \big \}$. Then by following Equation \eqref{eq;sigma algebra} and Definition \ref{def; lms}, we get
a locally measure space $\big (X = \mathbb{R}, B(\mathbb{R}), \mu \big )$. For each $p \in \mathbb{R}$, let $\big \{ \mathcal{H}_p = \text{L}^{2}(\mathbb{R}, \mu) ; \mathcal{E}_p = \{\mathcal{H}_{\alpha, p} = \text{L}^{2}(\mathbb{R}, \mu) \}_{\alpha \in [0, \infty)}; \mathcal{D}_p = \text{L}^{2}(\mathbb{R}, \mu)  \big \}$ be a  quantized domain. Note that here each $\mathcal{D}_{p} = \text{L}^{2}(\mathbb{R}, \mu)$ is indeed a Hilbert space. By following Equation \eqref{eq;dilhs=u}, we get
\begin{equation*}
\displaystyle \int^{\oplus_{\text{loc}}}_{\mathbb{R}_\text{loc}} \mathcal{D}_{p} \, \dmu = \displaystyle \int^{\oplus_{\text{loc}}}_{\mathbb{R}_\text{loc}} \text{L}^{2}(\mathbb{R}, \mu) \, \dmu = \bigcup\limits_{\alpha \in {[0, \infty)}} \mathcal{H}_\alpha,
\end{equation*}
where 
\begin{equation*}
\mathcal{H}_\alpha = \left \{ u : \mathbb{R} \rightarrow \text{L}^{2}(\mathbb{R}, \mu)  ~~ : ~~  \text{supp}(u) \subseteq [-\alpha, \alpha] \; \; \text{and} \; \; \int_{\mathbb{R}} \la u(p), u(p) \ra \; \dmu < \infty \right \},
\end{equation*}
which is isomorphic to the Hilbert space $\text{L}^2 \big ([\alpha, \alpha], \mu_\alpha \big) \otimes \text{L}^{2}(\mathbb{R}, \mu)$. Thus, one may see that the locally Hilbert space $\displaystyle \int^{\oplus_{\text{loc}}}_{\mathbb{R}_\text{loc}} \text{L}^{2}(\mathbb{R}, \mu) \, \dmu $ is dense in the Hilbert space $\displaystyle \int^{\oplus}_{\mathbb{R}} \text{L}^{2}(\mathbb{R}, \mu) \, \dmu$.
\end{example}

It is evident from the Example \ref{ex;dilhs direct integral of C} and Example \ref{ex;dilhs direct integral of L2 R mu} that even though each $\mathcal{D}_p$ is a Hilbert space, the space $\displaystyle \dilX \mathcal{D}_p \, \dmu$ need not be a Hilbert space. However, as shown in Proposition \ref{prop;dilhs}, $\displaystyle \dilX \mathcal{D}_p \, \dmu$ is always a locally Hilbert space. 


\section{Decomposable and Diagonalizable Locally Bounded Operators} \label{sec; Decomposable and Diagonalizable Locally Bounded Operators}

Let $(\Lambda, \leq)$ be a directed poset and $(X, \Sigma, \mu)$ be a locally measure space as described in Note \ref{note; lms}. Consider a family $\big \{ \mathcal{H}_p; \mathcal{E}_p = \{\mathcal{H}_{\alpha, p}\}_{\alpha \in \Lambda}; \mathcal{D}_p \big \}_{p \in X}$ of quantized domains. By following Remark \ref{rem; quntized domain}, we obtain  $\left \{ \overline{\displaystyle \dilX \mathcal{D}_p \, \dmu}; \mathcal{E} = \{\mathcal{H}_{\alpha}\}_{\alpha \in \Lambda}; \displaystyle \dilX \mathcal{D}_p \, \dmu \right \}$ a quantized domain. In this section, we turn our attention to subcollections of $C^\ast_{\mathcal{E}}\left (\displaystyle \dilX \mathcal{D}_p \, \dmu \right)$ that align with the structure of direct integrals. We introduce these classes motivated by the classical setup of a direct integral of Hilbert spaces.

\begin{definition} \label{def;DecDiag(lbo)}
Let $(X, \Sigma, \mu)$ be a locally measure space and $\big \{ \mathcal{H}_p; \mathcal{E}_p = \{\mathcal{H}_{\alpha, p}\}_{\alpha \in \Lambda}; \mathcal{D}_p \big \}_{p \in X}$ be a family of quantized domains. 
Then a locally bounded operator $T \in C^\ast_{\mathcal{E}}\left (\displaystyle \dilX \mathcal{D}_p \, \dmu \right)$ is said to be:
\begin{enumerate} 
\item \label{def;Dec(lbo)} \textbf{decomposable}, if there exists a family $ \big \{ T_p \in C^\ast_{\mathcal{E}_p}\left ( \mathcal{D}_p \right) \big \}_{p \in X}$ of locally bounded operators such that for any $u \in \displaystyle \dilX \mathcal{D}_p \, \dmu$, we have
\begin{align*}
(Tu)(p) = T_pu(p) \; \; \; \; \text{for} \; \; \mu\text{-a.e.}
\end{align*}
In this case, we denote the operator $T$ by the notation $\displaystyle \dilX T_p \, \dmu$ and so
\begin{equation*}
\left (\dilX T_p \, \dmu \right ) \left (\dilX u(p) \, \dmu \right )= \dilX T_pu(p) \, \dmu;
\end{equation*}
\item \label{def;Diag(lbo)} \textbf{diagonalizable}, if $T$ is decomposable and there exists a measurable function $f : X \rightarrow \mathbb{C}$ such that for any $u \in \displaystyle \dilX \mathcal{D}_p \, \dmu$, we have 
\begin{equation*}
(Tu)(p) = f(p)u(p) \; \; \; \; \text{for} \; \; \mu\text{-a.e.}
\end{equation*}
In this situation, we get $T = \displaystyle \dilX T_p \, \dmu = \displaystyle \dilX f(p) \cdot \mathrm{Id}_{\mathcal{D}_p} \, \dmu$.
\end{enumerate}  
\end{definition}

\noindent
We denote the collection of all decomposable locally bounded operators and the collection of all diagonalizable locally bounded operators on $\displaystyle \dilX \mathcal{D}_p \, \dmu$ by $C^\ast_{\mathcal{E}, \text{DEC}}\left (\displaystyle \dilX \mathcal{D}_p \, \dmu \right)$ and $C^\ast_{\mathcal{E}, \text{DIAG}}\left (\displaystyle \dilX \mathcal{D}_p \, \dmu \right)$ respectively. From Definition \ref{def;DecDiag(lbo)}, one may observe that
\begin{equation} \label{eqn; containment}
C^\ast_{\mathcal{E}, \text{DIAG}}\left (\displaystyle \dilX \mathcal{D}_p \, \dmu \right) \subseteq C^\ast_{\mathcal{E}, \text{DEC}}\left (\displaystyle \dilX \mathcal{D}_p \, \dmu \right) \subseteq C^\ast_{\mathcal{E}}\left (\displaystyle \dilX \mathcal{D}_p \, \dmu \right).
\end{equation}

Next, we give some examples of the notion defined above. 

\begin{example}
Consider $\big (\Lambda = \mathbb{N}, \leq \big )$ the directed poset, $\big \{(X_n = [-n, n], B(X_n), \mu_n) \big \}_{n \in \mathbb{N}}$ the family of measure spaces, and $\big ( X = \mathbb{R}, B(\mathbb{R}), \mu \big)$ the locally measure space as described in Example \ref{ex;lsms}. The  space $\displaystyle \displaystyle \int^{\oplus_\text{loc}}_{\mathbb{R}_\text{loc}} \mathbb{C} \, \dmu$ is described in Equation \eqref{eq; direct integral of C} of  Example \ref{ex;dilhs direct integral of C}, where we obtain
\begin{equation*}
\displaystyle \int^{\oplus_\text{loc}}_{\mathbb{R}_\text{loc}} \mathbb{C} \, \dmu = \Big\{ u \in \text{L}^2 \big (\mathbb{R}, \mu \big) \; \; : \; \; \text{supp}(u) \subseteq [-n, n] \; \; \text{for some} \; \; n \in \mathbb{N}  \Big\}.
\end{equation*}
Let $f : \mathbb{R} \rightarrow \mathbb{C}$ be a measurable function defined by $f(p) := p$ for all $p \in \mathbb{R}$. Corresponding to the function $f$, we obtain a locally bounded operator 
\begin{equation*}
T_f \in  C^\ast_{\mathcal{E}}\left (\displaystyle \displaystyle \int^{\oplus_\text{loc}}_{\mathbb{R}_\text{loc}} \mathbb{C} \, \dmu \right) \; \; \; \; \text{defined by} \; \; \; \; (T_fu)(p) := f(p)u(p) = pu(p), \; \; \;\text{for a.e.} \; \; p \in \mathbb{R},
\end{equation*}
for every $u \in  \displaystyle \displaystyle \int^{\oplus_\text{loc}}_{\mathbb{R}_\text{loc}} \mathbb{C} \, \dmu$. Then $T_f \in C^\ast_{\mathcal{E}, \text{DIAG}}\left (\displaystyle \displaystyle \int^{\oplus_\text{loc}}_{\mathbb{R}_\text{loc}} \mathbb{C} \, \dmu \right) \subseteq C^\ast_{\mathcal{E}, \text{DEC}}\left (\displaystyle \displaystyle \int^{\oplus_\text{loc}}_{\mathbb{R}_\text{loc}} \mathbb{C} \, \dmu \right)$. 
\end{example}

\noindent
Now we give an example of a decomposable locally bounded operator (defined on a direct integral of locally Hilbert spaces) that is not diagonalizable .

\begin{example} \label{eg; Dec but not Diag 1}
Consider $\big (\Lambda = \mathbb{N}, \leq \big )$ the directed poset. For each $i \in \mathbb{N}$, consider a measure space $\big \{ \big ( X_i = \{1, 2, 3, ..., i \}, \Sigma_i, \mu_i \big ) \big \}$ and a locally measure space $\big ( X = \mathbb{N}, \Sigma, \mu \big )$ as given in Example \ref{ex;dilhs direct sum and direct integral}. Let $\big \{ \mathcal{H}_n; \mathcal{E}_n = \{\mathcal{H}_{i, n}\}_{i \in \mathbb{N}}; \mathcal{D}_n \big \}_{n \in \mathbb{N}}$ be a family of quantized domains, where 
\begin{equation*}
\big \{ \mathcal{H}_1 = \ell^2(\mathbb{N}); \mathcal{E}_1 = \{\text{span} \{ e_1, e_2, ..., e_i \}\}_{i \in \mathbb{N}}; \mathcal{D}_1 \big \} \; \; \; \text{and} \; \; \; \big \{ \mathcal{H}_n  = \{ 0 \}; \mathcal{E}_n = \{\{ 0 \}\}_{i \in \mathbb{N}}; \mathcal{D}_n \big \}_{n \geq 2} 
\end{equation*}
\big (here $\{e_{n}: n \in \mathbb{N}\}$ is a Hilbert basis of $\ell^2(X = \mathbb{N})$ \big). Now consider the direct integral $\displaystyle\int^{\oplus_\text{loc}}_{\mathbb{N}_\text{loc}} \mathcal{D}_n \; \mathrm{d} \mu(n)$ of the family $\{ \mathcal{D}_n \}_{n \in \mathbb{N}}$. However, by recalling Example \ref{ex;dilhs direct sum and direct integral}, we get $\displaystyle\int^{\oplus_\text{loc}}_{\mathbb{N}_\text{loc}} \mathcal{D}_n \; \mathrm{d} \mu(n) = \bigoplus^\infty_{n = 1} \mathcal{D}_n$.
Now define a operator $T \in C^\ast_{\mathcal{E}}\left (\displaystyle\int^{\oplus_\text{loc}}_{\mathbb{N}_\text{loc}} \mathcal{D}_n \; \mathrm{d} \mu(n) \right)$ as 
\begin{equation*}
T\big ( \big \{ u(n) \big \}_{n \in \mathbb{N}}  \big ) = \big \{ T_nu(n) \big \}_{n \in \mathbb{N}},
\end{equation*}
where $T_1$ is as defined in Example \ref{ex;lbo} and $T_n = 0$ for $n \geq 2$. For instance, 
\begin{equation*} 
T \left ( \left \{ \sum^N_{k = 1} \lambda_k e_k, 0, 0, 0, 0, ... \right \} \right) = \left \{ \sum^N_{k = 1} k \lambda_k e_k, 0, 0, 0, 0, ... \right \}.
\end{equation*}
It shows that $T$ is a decomposable locally bounded operator on $\displaystyle\int^{\oplus_\text{loc}}_{\mathbb{N}_\text{loc}} \mathcal{D}_n \; \mathrm{d} \mu(n)$, that is, $T \in C^\ast_{\mathcal{E}, \text{DEC}}\left (\displaystyle\int^{\oplus_\text{loc}}_{\mathbb{N}_\text{loc}} \mathcal{D}_n \; \mathrm{d} \mu(n) \right)$.  Now we show that $T \notin C^\ast_{\mathcal{E}, \text{DIAG}}\left (\displaystyle\int^{\oplus_\text{loc}}_{\mathbb{N}_\text{loc}} \mathcal{D}_n \; \mathrm{d} \mu(n) \right)$. Let us assume $T \in C^\ast_{\mathcal{E}, \text{DIAG}}\left (\displaystyle\int^{\oplus_\text{loc}}_{\mathbb{N}_\text{loc}} \mathcal{D}_n \; \mathrm{d} \mu(n) \right)$, then by following (\ref{def;Diag(lbo)}) of Definition \ref{def;DecDiag(lbo)}, there is a measurable function $f : \mathbb{N} \rightarrow \mathbb{C}$ satisfying $T\big ( \big \{ u(n) \big \}_{n \in \mathbb{N}}  \big ) = \big \{ f(n)u(n) \big \}_{n \in \mathbb{N}}$ for every $\big \{ u(n) \big \}_{n \in \mathbb{N}} \in \bigoplus\limits^\infty_{n = 1} \mathcal{D}_n$. Then for all $N \in \mathbb{N}$, we get 
\begin{equation} \label{eq; T is Dec not Diag}
\left \{ f(1) \left ( \sum^N_{k = 1}\lambda_k e_k \right ), 0, 0, 0, 0, ... \right \} = T \left( \left \{  \sum^N_{k = 1}\lambda_k e_k, 0, 0, 0, 0,... \right \} \right) = \left \{  \sum^N_{k = 1}k \lambda_k e_k, 0, 0, 0, 0, ... \right \},
\end{equation}
which is a contradiction. 
\end{example}

We furnish the following example with the intention that such a decomposable, non diagonalizable locally bounded operator exists even when $\Lambda$ is an arbitrary (possibly uncountable) directed poset.

\begin{example} \label{eg; Dec but not Diag 2}
Consider a directed poset $\big ( \Lambda = [0, \infty), \leq \big )$. For each $\alpha \in [0, \infty)$ consider a measure space $\big \{ \big ( X_\alpha = [-\alpha, \alpha], B \big ([-\alpha, \alpha] \big), \mu_\alpha \big ) \big \}$. Then as discussed in Example \ref{ex;dilhs direct integral of L2 R mu}, we get
a locally measure space $\big (X = \mathbb{R}, B(\mathbb{R}), \mu \big )$. Let $\big \{ \mathcal{H}_p; \mathcal{E}_p = \{\mathcal{H}_{\alpha, p} = \text{L}^{2}(\mathbb{R}, \mu) \}_{\alpha \in [0, \infty)}; \mathcal{D}_p \big \}_{p \in \mathbb{R}}$ be a family of quantized domains. By following Example \ref{ex;dilhs direct integral of L2 R mu}, we obtain 
\begin{equation*}
\displaystyle \int^{\oplus_{\text{loc}}}_{\mathbb{R}_\text{loc}} \text{L}^{2}(\mathbb{R}, \mu) \, \dmu = \bigcup\limits_{\alpha \in {[0, \infty)}} \mathcal{H}_\alpha,
\end{equation*}
where 
\begin{equation*}
\mathcal{H}_\alpha = \left \{ u : \mathbb{R} \rightarrow \text{L}^{2}(\mathbb{R}, \mu)  ~~ : ~~  \text{supp}(u) \subseteq [-\alpha, \alpha] \; \; \text{and} \; \; \int_{\mathbb{R}} \la u(p), u(p) \ra \; \dmu < \infty \right \}.
\end{equation*}
Let us define an operator $T : \displaystyle \int^{\oplus_{\text{loc}}}_{\mathbb{R}_\text{loc}} \text{L}^{2}(\mathbb{R}, \mu) \, \dmu \rightarrow \displaystyle \int^{\oplus_{\text{loc}}}_{\mathbb{R}_\text{loc}} \text{L}^{2}(\mathbb{R}, \mu) \, \dmu$ as 
\begin{equation*}
T \left ( u = \displaystyle \int^{\oplus_{\text{loc}}}_{\mathbb{R}_\text{loc}} u(p) \, \dmu \right ) := \displaystyle \int^{\oplus_{\text{loc}}}_{\mathbb{R}_\text{loc}} \hat{u}(p) \, \dmu,
\end{equation*}
where $\hat{u}(p) : \mathbb{R} \rightarrow \mathbb{C}$ is defined as $\hat{u}(p)(t) := u(p)(2t)$ for every $t \in \mathbb{R}$.
Since $\text{supp}(u) = \text{supp}(Tu)$ for every $u \in \displaystyle \int^{\oplus_{\text{loc}}}_{\mathbb{R}_\text{loc}}  \text{L}^{2}(\mathbb{R}, \mu) \, \dmu$ it follows that each $\mathcal{H}_\alpha$ is a reducing subspace for $T$. Hence $T \in C^\ast_{\mathcal{E}}\left (\displaystyle \int^{\oplus_{\text{loc}}}_{\mathbb{R}_\text{loc}}  \text{L}^{2}(\mathbb{R}, \mu) \, \dmu \right)$. Further, $T \in C^\ast_{\mathcal{E}, \text{DEC}}\left (\displaystyle \int^{\oplus_{\text{loc}}}_{\mathbb{R}_\text{loc}}  \text{L}^{2}(\mathbb{R}, \mu) \, \dmu \right)$. Because there is a family $\big \{ T_p \in C^\ast_{\mathcal{E}_p}\left (\mathcal{D}_p =  \text{L}^{2}(\mathbb{R}, \mu) \right) \big \}_{p \in X}$ given by $T_p : \text{L}^{2}(\mathbb{R}, \mu) \rightarrow \text{L}^{2}(\mathbb{R}, \mu)$, where $T_p(f)(t) = f(2t)$ for every $f \in \text{L}^{2}(\mathbb{R}, \mu)$ and $t, p \in \mathbb{R}$ such that for each $u \in \displaystyle \int^{\oplus_{\text{loc}}}_{\mathbb{R}_\text{loc}}  \text{L}^{2}(\mathbb{R}, \mu) \, \dmu$, we have
\begin{equation*}
(Tu)(p) = T_p(u(p))  \; \; \; \; \text{for a.e.}  \; p \in X.
\end{equation*}
Now we show that $T \notin C^\ast_{\mathcal{E}, \text{DIAG}}\left (\displaystyle \int^{\oplus_{\text{loc}}}_{\mathbb{R}_\text{loc}}  \text{L}^{2}(\mathbb{R}, \mu) \, \dmu \right)$. Let $u : \mathbb{R} \rightarrow \text{L}^{2}(\mathbb{R}, \mu)$ be given by 
\begin{equation*}
u(p)(t) := \begin{cases}
1, & \text{if} \;\; t \in [-10, 10] \;,\; p \in [-1, 1];\\
0 & \text{otherwise}.
\end{cases}
\end{equation*}
Then $u \in \displaystyle \int^{\oplus_{\text{loc}}}_{\mathbb{R}_\text{loc}}  \text{L}^{2}(\mathbb{R}, \mu) \, \dmu$. Suppose $T \in C^\ast_{\mathcal{E}, \text{DIAG}}\left (\displaystyle \int^{\oplus_{\text{loc}}}_{\mathbb{R}_\text{loc}}  \text{L}^{2}(\mathbb{R}, \mu) \, \dmu \right)$, then by following (\ref{def;Diag(lbo)}) of Definition \ref{def;DecDiag(lbo)}, there exists a measurable function $f : \mathbb{R} \rightarrow \mathbb{C}$ such that 
\begin{equation} \label{Eq: Suppose T is diag}
(Tu)(p) = f(p)u(p),\; \text{for a.e.}\; p \in \mathbb{R}. 
\end{equation}
By the definition of $T$, we get $Tu(p)(t) = 1$, whenever $p \in [-1, 1]$ and $t \in [-5, 5]$. From Equation \eqref{Eq: Suppose T is diag}, it follows that $f(p) = 1$ for a.e. $p \in [-1, 1]$. On the other hand $Tu(p)(t) = 0$, whenever $p \in [-1, 1]$ and $t \in [-10, -5) \cup (5, 10]$. That is, $f(p) = 0$ for a.e. $p \in [-1, 1]$ (from Equation \eqref{Eq: Suppose T is diag}). This is a contradiction. Therefore, $T \notin C^\ast_{\mathcal{E}, \text{DIAG}}\left (\displaystyle \int^{\oplus_{\text{loc}}}_{\mathbb{R}_\text{loc}}  \text{L}^{2}(\mathbb{R}, \mu) \, \dmu \right)$.
\end{example}

Next, we present an example of a locally bounded operator which is not decomposable.

\begin{example} \label{eg; LBO but not Dec}
Let $\big ( \Lambda = \{ 1 \} \cup [2, \infty), \leq \big )$ be directed poset. For each $\alpha \in \Lambda$ consider a measure space $\big \{ \big ( X_\alpha = [-\alpha, \alpha], B \big ([-\alpha, \alpha] \big), \mu_\alpha \big ) \big \}$. Then by following Equation \eqref{eq;sigma algebra} and Definition \ref{def; lms}, we get
a locally measure space $\big (X = \mathbb{R}, B(\mathbb{R}), \mu \big )$. Let $\big \{ \mathcal{H}_p; \mathcal{E}_p = \{\mathcal{H}_{\alpha, p} = \mathbb{C}\}_{\alpha \in \Lambda}; \mathcal{D}_p \big \}_{p \in \mathbb{R}}$ be a family of quantized domains. Then for each $p \in X$, we get a locally Hilbert space $\mathcal{D}_{p} = \varinjlim\limits_{\alpha \in \Lambda} \mathcal{H}_{\alpha, p} = \mathbb{C}$. Note that here each $\mathcal{D}_{p}$ is indeed a Hilbert space and for any $p, q \in \mathbb{R}$, we have $\mathcal{H}_{p} = \mathcal{D}_{p} = \mathbb{C} = \mathcal{D}_{q} = \mathcal{H}_{q}$. By following the similar procedure as in Example \ref{ex;dilhs direct integral of C}, we get 
\begin{equation*}
\int^{\oplus_\text{loc}}_{\mathbb{R}_\text{loc}} \mathbb{C} \, \dmu 
 = \Big\{ u \in \text{L}^2 \big (\mathbb{R}, \mu \big) \; \; : \; \; \text{supp}(u) \subseteq [-\alpha, \alpha] \; \; \text{for some} \; \; \alpha \in \Lambda  \Big\},
\end{equation*}
In fact, $\displaystyle \int^{\oplus_\text{loc}}_{\mathbb{R}_\text{loc}} \mathbb{C} \, \dmu  = \bigcup\limits_{\alpha \in \Lambda} \mathcal{H}_{\alpha}$ (see Equation \eqref{eq;dilhs=u}), where 
\begin{equation*}
\mathcal{H}_\alpha := \left \{ u \in \text{L}^2 \big (\mathbb{R}, \mu \big) \; \; : \; \; \text{supp}(u) \subseteq [-\alpha, \alpha] \right \}.
\end{equation*}
Now we define an operator $T : \displaystyle \int^{\oplus_\text{loc}}_{\mathbb{R}_\text{loc}} \mathbb{C} \, \dmu \rightarrow \displaystyle \int^{\oplus_\text{loc}}_{\mathbb{R}_\text{loc}} \mathbb{C} \, \dmu$ by 
\begin{equation*}
T(u)(p) := \up{\chi}_{[\frac{-1}{2}, \frac{1}{2}]}(p)\; u(2p), \; \; \; \text{for every} \; \; u \in \int^{\oplus_\text{loc}}_{\mathbb{R}_\text{loc}} \mathbb{C} \, \dmu.
\end{equation*}
It is immediate to see that $\text{supp}(Tu) \subseteq [\frac{-1}{2}, \frac{1}{2}]$. Thus, for every $\alpha \in \Lambda$,  $T(\mathcal{H}_{\alpha}) \subseteq \mathcal{H}_{1} \subseteq \mathcal{H}_\alpha$ and so, $T \in C^\ast_{\mathcal{E}}\left (\displaystyle \int^{\oplus_\text{loc}}_{\mathbb{R}_\text{loc}} \mathbb{C} \, \dmu  \right)$. Now we show that $T \notin C^\ast_{\mathcal{E}, \text{DEC}}\left (\displaystyle \int^{\oplus_\text{loc}}_{\mathbb{R}_\text{loc}} \mathbb{C} \, \dmu  \right)$.  Suppose $T$ is decomposable, then by (\ref{def;Dec(lbo)}) of Definition \ref{def;DecDiag(lbo)}, there exists a family $\big \{ T_p \in C^\ast_{\mathcal{E}_p}\left (\mathcal{D}_p = \mathbb{C} \right) \big \}_{p \in \mathbb{R}}$ satisfying for every $u \in \displaystyle \int^{\oplus_\text{loc}}_{\mathbb{R}_\text{loc}} \mathbb{C} \, \dmu$
\begin{equation*}
(Tu)(p) = T_pu(p), \; \; \; \text{a.e.} \; \;  p \in \mathbb{R}.
\end{equation*}
Since each $T_p$ is a linear operator on $\mathbb{C}$, we obtain $T_p = c_p$ for some $c_p \in \mathbb{C}$. Now, if we take
\begin{equation*}
u(p) := \begin{cases}
p, & \text{if} \;\; p \in [\frac{-1}{2}, \frac{1}{2}]; \\
0 & \text{otherwise}
\end{cases} \; \; \; \; \; \; 
v(p) := \begin{cases}
p, & \text{if} \;\; p \in [\frac{-1}{4}, \frac{1}{4}]; \\
0 & \text{otherwise},
\end{cases}
\end{equation*}
then for a.e. $p \in [\frac{-1}{4}, \frac{-1}{8}] \cup [\frac{1}{8}, \frac{1}{4}]$, we get
\begin{equation*}
T_pu(p) = c_p p = (Tu)(p) = 2p,
\end{equation*}
that is $c_p = 2$. Whereas, for a.e. $p \in [\frac{-1}{4}, \frac{-1}{8}] \cup [\frac{1}{8}, \frac{1}{4}]$, we obtain
\begin{equation*}
T_pv(p) = c_p p = (Tv)(p) = 0,
\end{equation*}
that is $c_p = 0$. This is a contradiction. Therefore, $T \notin C^\ast_{\mathcal{E}, \text{DEC}}\left (\displaystyle \int^{\oplus_\text{loc}}_{\mathbb{R}_\text{loc}} \mathbb{C} \, \dmu  \right)$. 
\end{example}

\subsection{Observations I} \label{obs; 1} Let $(\Lambda, \leq)$ be a directed poset and $(X, \Sigma, \mu)$ be a locally measure space as described in Note \ref{note; lms}. Consider a family $\big \{ \mathcal{H}_p; \mathcal{E}_p = \{\mathcal{H}_{\alpha, p}\}_{\alpha \in \Lambda}; \mathcal{D}_p \big \}_{p \in X}$ of quantized domains. By following Remark \ref{rem; quntized domain}, we obtain  $\left \{ \overline{\displaystyle \dilX \mathcal{D}_p \, \dmu}; \mathcal{E} = \{\mathcal{H}_{\alpha}\}_{\alpha \in \Lambda}; \displaystyle \dilX \mathcal{D}_p \, \dmu \right \}$ a quantized domain.  The following key observations are useful in understanding the notion of  decomposable and diagonalizable locally bounded operators on $\displaystyle \dilX \mathcal{D}_p \, \dmu$.

\begin{enumerate} 
\item Let $T \in C^\ast_{\mathcal{E}, \text{DEC}}\left (\displaystyle \dilX \mathcal{D}_p \, \dmu \right)$. Then unique up to a measure zero set there exists a family $\big \{T_p \in C^\ast_{\mathcal{E}_p}\left (\mathcal{D}_p \right) \big \}_{p \in X}$ such that $T = \displaystyle \dilX T_p \, \dmu$. Suppose $\big \{T_p \in C^\ast_{\mathcal{E}_p}\left (\mathcal{D}_p \right) \big \}_{p \in X}$ and $\big \{T^\prime_p \in C^\ast_{\mathcal{E}_p}\left (\mathcal{D}_p \right) \big \}_{p \in X}$ are two distinct families such that $\displaystyle \dilX T_p \, \dmu = T = \displaystyle \dilX T^\prime_p \, \dmu$. Then from point (3) of Definition \ref{Defn: directint_loc}, we get $T_p = T^\prime_p$ for $\mu$-a.e.

\item \label{obs;M DEC and M DIAG are star algebras}
From Example \ref{ex;lca}, we know that the collection of all locally bounded operators on $\displaystyle \dilX \mathcal{D}_p \, \dmu$  denoted by $C^\ast_\mathcal{E}\left(\displaystyle \dilX \mathcal{D}_p \, \dmu \right)$ is a locally $C^\ast$-algebra. From Equation \eqref{eqn; containment}, we have 
$C^\ast_{\mathcal{E}, \text{DIAG}}\left (\displaystyle \dilX \mathcal{D}_p \, \dmu \right) \subseteq C^\ast_{\mathcal{E}, \text{DEC}}\left (\displaystyle \dilX \mathcal{D}_p \, \dmu \right) \subseteq C^\ast_{\mathcal{E}}\left (\displaystyle \dilX \mathcal{D}_p \, \dmu \right)$. We get $C^\ast_{\mathcal{E}, \text{DEC}}\left (\displaystyle \dilX \mathcal{D}_p \, \dmu \right)$ to be a locally convex $\ast$-subalgebra of $C^\ast_\mathcal{E}\left(\displaystyle \dilX \mathcal{D}_p \, \dmu \right)$ with respect to the following operations
\begin{multicols}{2}
\begin{enumerate}
\item $T + S = \displaystyle \dilX T_p + S_p \, \dmu$\\
\item $\lambda \cdot T = \displaystyle \dilX \lambda \cdot T_p\, \dmu$
\item $T \cdot S = \displaystyle \dilX T_p \cdot S_p \, \dmu$\\
\item $T^\ast = \displaystyle \dilX T^\ast_p \, \dmu$,
\end{enumerate}
\end{multicols}

\noindent
for every $T = \displaystyle \dilX T_p \, \dmu$ and $S = \displaystyle \dilX S_p \, \dmu$ in $C^\ast_\mathcal{E}\left(\displaystyle \dilX \mathcal{D}_p \, \dmu \right)$ and $\lambda \in \mathbb{C}$. 
Also, with respect to the similar operations $C^\ast_{\mathcal{E}, \text{DIAG}}\left (\displaystyle \dilX \mathcal{D}_p \, \dmu \right)$ forms a locally convex $\ast$-subalgebra of $C^\ast_\mathcal{E}\left(\displaystyle \dilX \mathcal{D}_p \, \dmu \right)$.

\item \label{obs; description of V alpha T V alpha star}
Suppose $T \in C^\ast_{\mathcal{E}, \text{DEC}}\left (\displaystyle \dilX \mathcal{D}_p \, \dmu \right)$, then there is a family $\big \{T_p \in C^\ast_{\mathcal{E}_p}\left (\mathcal{D}_p \right) \big \}_{p \in X}$ such that $T = \displaystyle \dilX T_p \, \dmu$. Also, we have $\displaystyle \dilX \mathcal{D}_p \, \dmu = \bigcup\limits_{\alpha \in \Lambda} \mathcal{H}_\alpha$ (see Equation \eqref{eq;dilhs=u}),
where 
\begin{equation*}
\mathcal{H}_\alpha := \left \{ u \in \dilX \mathcal{D}_p \, \dmu ~~ : ~~   u(p) \in \mathcal{H}_{\alpha, p} \; \; \text{for a.e.} \; \; p \in X_\alpha \; \; \text{and} \; \; \text{supp}(u) \subseteq X_\alpha \right \}.
\end{equation*}
\noindent
Now for each $\alpha \in \Lambda$ consider the isomorphism $V_\alpha : \mathcal{H}_\alpha \rightarrow \int^\oplus_{X_\alpha} \mathcal{H}_{\alpha, p} \, \mathrm{d} \mu_\alpha (p)$ given by $V_\alpha(u)(p) := u(p)$ for all $p \in X_\alpha$ and $u \in \mathcal{H}_\alpha$ \big (refer Equation \eqref{eq;iso} \big ). Fix $\alpha \in \Lambda$ and  $V_\alpha (u) \in \int^\oplus_{X_\alpha} \mathcal{H}_{\alpha, p} \, \mathrm{d} \mu_\alpha (p)$ for some $u \in \mathcal{H}_\alpha$, then 
\begin{equation*}
V_\alpha T V^\ast_\alpha (V_\alpha(u))(p) =  V_\alpha (Tu)(p) = (Tu)(p) = T_p u(p) = T_p \big |_{\mathcal{H}_{\alpha, p}} u(p),
\end{equation*}
for a.e. $p \in X_\alpha$. Thus $V_\alpha T V_\alpha^\ast$ is a decompoable bounded operator on $\int^\oplus_{X_\alpha} \mathcal{H}_{\alpha, p} \, \mathrm{d} \mu_\alpha (p)$, and 
\begin{equation} \label{eq;restriction of DecLBO}
V_\alpha T V_\alpha^\ast = \int^\oplus_{X_\alpha} T_{p} \big|_{\mathcal{H}_{\alpha, p}} \,  \mathrm{d} \mu_\alpha (p) \; \; \; \text{for every} \; \; \alpha \in \Lambda.
\end{equation}
Moreover, for each $\alpha \in \Lambda$, we obtain 
\begin{equation}\label{eq; norm of T restricted to H alpha}
\big \| T\big|_{\mathcal{H}_\alpha} \big \| =  \big \| V_\alpha T V_\alpha^\ast \big \| = \text{ess} \sup\limits_{p \in X_\alpha} \big \{ \big \| T_{p}\big|_{\mathcal{H}_{\alpha, p}} \big\| \big \} < \infty.  
\end{equation}

In particular, if $T \in C^\ast_{\mathcal{E}, \text{DIAG}}\left (\displaystyle \dilX \mathcal{D}_p \, \dmu \right)$ with 
$T = \displaystyle \dilX f(p) \cdot \mathrm{Id}_{\mathcal{D}_{p}} \, \dmu$, for some measurable function $f : X \rightarrow \mathbb{C}$, then for each $\alpha \in \Lambda$, we have 
\begin{equation} \label{eq;restriction of DiagLBO}
V_\alpha T V_\alpha^\ast = \int^\oplus_{X_\alpha} f(p) \cdot \mathrm{Id}_{\mathcal{H}_{\alpha, p}} \,  \mathrm{d} \mu_\alpha (p),
\end{equation}
where $f \big|_{X_\alpha} \in \text{L}^\infty \big (X_\alpha, \mu_\alpha \big )$. As a result
$V_\alpha T V_\alpha^\ast$ is a diagonalizable bounded linear operator for each $\alpha \in \Lambda$.

\item If $T \in C^\ast_{\mathcal{E}, \text{DIAG}}\left (\displaystyle \dilX \mathcal{D}_p \, \dmu \right)$, then
$T = \displaystyle \dilX f(p) \cdot \mathrm{Id}_{\mathcal{D}_p}  \, \dmu$, for some measurable function $f : X \rightarrow \mathbb{C}$ and from the previous observation, we get that for each $\alpha \in \Lambda$, the bounded linear operator $V_\alpha T V_\alpha^\ast$ on $\int^\oplus_{X_\alpha} \mathcal{H}_{\alpha, p} \, \mathrm{d} \mu_\alpha$ is diagonalizable. By following (\ref{def;Diagbo}) of Definition \ref{def;Debo}, corresponding to each $V_\alpha T V_\alpha^\ast$ there is a function $f_\alpha \in \text{L}^\infty \big (X_\alpha, \mu_\alpha \big )$. As a result, to define diagonalizable locally bounded operator, one may think of considering the family of measurable functions $ \big \{ f_\alpha  \in \text{L}^\infty \big (X_\alpha, \mu_\alpha \big ) \big \}_{\alpha \in \Lambda}$ such that
\begin{equation*}
(Tu)(p) = f_\alpha(p)u(p)  \; \; \text{for a.e.} \; \; p \in X_\alpha \; \; \text{and for every} \; \; u \in \mathcal{H}_\alpha.
\end{equation*}
In that case, by using the fact that $\mathcal{H}_\alpha \subseteq \mathcal{H}_\beta$ (whenever $\alpha \leq \beta)$, we see that  
\begin{equation*}
f_\alpha(p)u(p) = (Tu)(p)  = f_\beta(p)u(p)
\end{equation*}
for a.e. $p \in X_\alpha$ and for every $u \in \mathcal{H}_\alpha$. Thus, $f_\alpha(p) = f_\beta(p)$ for a.e. $p \in X_\alpha$. On the other hand, if $\alpha, \beta \in \Lambda$ are not comparable, then there exists $\gamma \in \Lambda$ such that $\alpha \leq \gamma$ and $\beta \leq \gamma$. For any $u \in \mathcal{H}_\alpha \subseteq \mathcal{H}_\gamma$ and $v \in \mathcal{H}_\beta \subseteq \mathcal{H}_\gamma$, we have
\begin{align*}
f_\alpha(p)u(p) = (Tu)(p)  &= f_\gamma(p)u(p) \; \; \text{for a.e.} \; \; p \in X_\alpha;  \\
f_\beta(q)v(q) = (Tv)(q)  &= f_\gamma(q)v(q) \; \; \text{for a.e.} \; \; q \in X_\beta.
\end{align*}  
Consequently, we get $f_\alpha(p) = f_\gamma(p) = f_\beta(p)$ for a.e. $p \in X_\alpha \cap X_\beta \cap X_\gamma$. 

Therefore, this shows that for any $\alpha, \beta \in \Lambda$, if $X_\alpha \cap X_\beta \neq \emptyset$, then $f_\alpha(p) = f_\beta(p)$ for a.e. $p \in X_\alpha \cap X_\beta$. In view of this, by defining $f : X \rightarrow \mathbb{C}$ by $f(p) := f_\alpha(p)$, whenever $p \in X_\alpha$, we get $f$ to be measurable such that $f \big |_{X_\alpha} \in  \text{L}^\infty \big (X_\alpha, \mu_\alpha \big )$ for every $\alpha \in \Lambda$ and $(Tu)(p) = f(p) u(p)$ for a.e. $p \in X$. Therefore, considering such family $ \big \{ f_\alpha  \in \text{L}^\infty \big (X_\alpha, \mu_\alpha \big ) \big \}_{\alpha \in \Lambda}$  is equivalent to saying that there is a measurable function $f : X \rightarrow \mathbb{C}$ as defined in \ref{def;Diag(lbo)} of Definition \ref{def;DecDiag(lbo)}.
\end{enumerate}

Next, we consider some specific cases in which the collection of all decomposable and the collection of all diagonalizable locally bounded operators on the direct integral of locally Hilbert spaces are locally von Neumann algebras.

\begin{theorem} \label{thm;DEC and DIAG LvNA}
Let $(\Lambda, \leq)$ be a directed poset and $(X, \Sigma, \mu)$ be a locally measure space (see Note \ref{note; lms}). For each $p \in X$ assign a quantized domain $\big \{ \mathcal{H}_p; \mathcal{E}_p = \{\mathcal{H}_{\alpha, p}\}_{\alpha \in \Lambda}; \mathcal{D}_p \big \}$. Suppose either $\Lambda$ is a countable set or $\mu$ is a countable measure on $X$, then 
\begin{enumerate}
\item[(a)] $C^\ast_{\mathcal{E}, \text{DEC}} \left(\displaystyle \dilX \mathcal{D}_p \, \dmu \right)$ is a locally von Neumann algebra;
\item[(b)] $C^\ast_{\mathcal{E}, \text{DIAG}} \left(\displaystyle \dilX \mathcal{D}_p \, \dmu \right)$ is an abelian locally von Neumann algebra.
\end{enumerate}
\end{theorem}
\begin{proof}
For each $\beta \in \Lambda$, consider the set $\Lambda_\beta = \{ \alpha \in \Lambda \;  :  \; \alpha \leq \beta \}$, the branch of $\Lambda$ determined by $\beta$. Then $(\Lambda_\beta, \leq)$ is a directed poset (see \cite[Section 1.4]{AG}). Let $\beta \in \Lambda$ be fixed. For each $p \in X_\beta$, $\mathcal{H}_{\beta, p} = \bigcup\limits_{\alpha \in \Lambda_\beta} \mathcal{H}_{\alpha, p}$, and so $\{ \mathcal{H}_{\beta, p}; \mathcal{E}_{\beta, p} = \{\mathcal{H}_{\alpha, p}\}_{\alpha \in \Lambda_\beta};  \mathcal{H}_{\beta, p} \big \}$ is a quantized domain. Associated with each $\alpha \in \Lambda_\beta$, we define a closed subspace $\mathcal{H}^{(\alpha)}_{\beta}$ of the Hilbert space $\int^{\oplus}_{X_\beta} \mathcal{H}_{\beta, p} \, \mathrm{d} \mu_\beta$ as 
\begin{equation}
\mathcal{H}^{(\alpha)}_{\beta} := \left \{ u \in \int^{\oplus}_{X_\beta} \mathcal{H}_{\beta, p} \, \mathrm{d} \mu_\beta ~~ : ~~  u(p) \in \mathcal{H}_{\alpha, p} \; \; \text{for a.e.} \; \; p \in X_\alpha  \; \; \; \text{and} \; \;  \text{supp}(u) \subseteq X_\alpha \right \}.
\end{equation}
Note that the family $\mathcal{E}_{\beta} = \{\mathcal{H}^{(\alpha)}_{\beta} \}_{\alpha \in \Lambda_\beta}$ is a strictly inductive system of Hilbert spaces with the union $\bigcup\limits_{\alpha \in \Lambda_\beta} \mathcal{H}^{(\alpha)}_{\beta} = \int^{\oplus}_{X_\beta} \mathcal{H}_{\beta, p} \, \mathrm{d} \mu_\beta$, that is, $\left \{ \int^{\oplus}_{X_\beta} \mathcal{H}_{\beta, p} \, \mathrm{d} \mu_\beta; \mathcal{E}_{\beta} = \{\mathcal{H}^{(\alpha)}_{\beta} \}_{\alpha \in \Lambda_\beta};  \int^{\oplus}_{X_\beta} \mathcal{H}_{\beta, p} \, \mathrm{d} \mu_\beta \right \}$ is a quantized domain. Following the notations, $C^\ast_{\mathcal{E}_\beta} \left( \int^{\oplus}_{X_\beta} \mathcal{H}_{\beta, p} \, \mathrm{d} \mu_\beta \right)$ consists of all bounded operators on the Hilbert space $\int^{\oplus}_{X_\beta} \mathcal{H}_{\beta, p} \, \mathrm{d} \mu_\beta$ for which the Hilbert subspace $\mathcal{H}^{(\alpha)}_{\beta}$ is a reducing subspace, whenever $\alpha \in \Lambda_\beta$. In particular, $C^\ast_{\mathcal{E}_\beta, \text{DEC}} \left( \int^{\oplus}_{X_\beta} \mathcal{H}_{\beta, p} \, \mathrm{d} \mu_\beta \right)$ and $C^\ast_{\mathcal{E}_\beta, \text{DIAG}} \left( \int^{\oplus}_{X_\beta} \mathcal{H}_{\beta, p} \, \mathrm{d} \mu_\beta \right)$ consists of all decomposable and diagonalizable bounded operators in $C^\ast_{\mathcal{E}_\beta} \left( \int^{\oplus}_{X_\beta} \mathcal{H}_{\beta, p} \, \mathrm{d} \mu_\beta \right)$ respectively.

\noindent
\textbf{Claim 1:} For each  $\beta\in \Lambda$, \; $C^\ast_{\mathcal{E}_\beta, \text{DEC}} \left( \int^{\oplus}_{X_\beta} \mathcal{H}_{\beta, p} \, \mathrm{d} \mu_\beta \right)$ is a von Neumann algebra in $\mathcal{B}\left( \int^{\oplus}_{X_\beta} \mathcal{H}_{\beta, p} \, \mathrm{d} \mu_\beta \right)$. 

From \ref{obs;M DEC and M DIAG are star algebras} of Observations \ref{obs; 1}, we know that this space is a $\ast$-algebra. Thus to prove our claim, it suffices to show that it is closed in the strong operator topology. To see this, let $\left \{ T_n \right \}_{n \in \mathbb{N}}$ be a sequence in $C^\ast_{\mathcal{E}_\beta, \text{DEC}} \left( \int^{\oplus}_{X_\beta} \mathcal{H}_{\beta, p} \, \mathrm{d} \mu_\beta \right)$ such that $T_n \rightarrow T$ in the strong operator topology on $\mathcal{B}\left( \int^{\oplus}_{X_\beta} \mathcal{H}_{\beta, p} \, \mathrm{d} \mu_\beta \right)$. Here $T$ is decomposable since  $C^\ast_{\mathcal{E}_\beta, \text{DEC}} \left(\int^{\oplus}_{X_\beta} \mathcal{H}_{\beta, p} \, \dmu \right)$ is contained in the von Neumann algebra of all decomposable operators on the Hilbert space $\int^{\oplus}_{X_\beta} \mathcal{H}_{\beta, p} \, \dmu$, and thus $T =  \int^{\oplus}_{X_\beta} T_{p} \, \dmu$. Now we show that for every $\alpha \in \Lambda_\beta$ the subspace $\mathcal{H}^{(\alpha)}_{\beta}$ is reducing for $T$. Fix $\alpha \in \Lambda_\beta$ and let $x \in \mathcal{H}^{(\alpha)}_{\beta}$. Then $T_nx \in \mathcal{H}^{(\alpha)}_{\beta}$ for each $n \in \mathbb{N}$, and so $Tx \in \mathcal{H}^{(\alpha)}_{\beta}$, since $\big \|  T_nx - Tx \big \|_{\mathcal{H}^{(\alpha)}_{\beta}} \rightarrow 0$, as $n \to \infty$. Therefore, the subspace $\mathcal{H}^{(\alpha)}_{\beta}$ is reducing for $T$ for every $\alpha \in \Lambda_\beta$ and hence $T \in C^\ast_{\mathcal{E}_\beta, \text{DEC}} \left( \int^{\oplus}_{X_\beta} \mathcal{H}_{\beta, p} \, \mathrm{d} \mu_\beta \right)$, which proves our Claim 1.

\noindent
\textbf{Claim 2:} For each  $\beta\in \Lambda$, \; $C^\ast_{\mathcal{E}_\beta, \text{DIAG}} \left( \int^{\oplus}_{X_\beta} \mathcal{H}_{\beta, p} \, \mathrm{d} \mu_\beta \right)$ is an abelian von Neumann algebra in $\mathcal{B}\left( \int^{\oplus}_{X_\beta} \mathcal{H}_{\beta, p} \, \mathrm{d} \mu_\beta \right)$. 

For a given $g \in \text{L}^\infty(X_\beta, \mu_\beta)$,  we define an operator $T_g$ in $C^\ast_{\mathcal{E}_\beta, \text{DIAG}} \left( \int^{\oplus}_{X_\beta} \mathcal{H}_{\beta, p} \, \mathrm{d} \mu_\beta \right)$ as   $T_g := \diX g(p) \cdot \mathrm{Id}_{\mathcal{H}_{\beta, p}}  , \mathrm{d} \mu_\beta$. Conversely, for $T \in C^\ast_{\mathcal{E}_\beta, \text{DIAG}} \left( \int^{\oplus}_{X_\beta} \mathcal{H}_{\beta, p} \, \mathrm{d} \mu_\beta \right)$, then there exists a function $f \in \text{L}^\infty(X_\beta, \mu_\beta)$ such that  $T = T_f = \diX f(p) \cdot \mathrm{Id}_{\mathcal{H}_{\beta, p}}  , \mathrm{d} \mu_\beta$  (see (\ref{def;Diagbo}) of Definition \ref{def;Debo}). Thus $C^\ast_{\mathcal{E}_\beta, \text{DIAG}} \left( \int^{\oplus}_{X_\beta} \mathcal{H}_{\beta, p} \, \mathrm{d} \mu_\beta \right)$ corresponds to the space $\text{L}^\infty(X_\beta, \mu_\beta)$. Then Claim 2 follows from \cite[Part II, Chapter 2, Section 4]{DixV}.



\noindent \textbf{Claim 3:} $C^\ast_{\mathcal{E}, \text{DEC}} \left(\displaystyle \dilX \mathcal{D}_p \, \dmu \right) = \varprojlim\limits_{\alpha \in \Lambda} C^\ast_{\mathcal{E}_\alpha, \text{DEC}} \left( \int^{\oplus}_{X_\alpha} \mathcal{H}_{\alpha, p} \, \mathrm{d} \mu_\alpha \right )$.

Let $T = \int^\oplus_{X_\beta} T_{\beta, p} \, \mathrm{d} \mu_\beta \in  C^\ast_{\mathcal{E}_\beta, \text{DEC}} \left(\int^{\oplus}_{X_\beta} \mathcal{H}_{\beta, p} \, \mathrm{d} \mu_\beta \right)$. Then $T_{\beta, p} : \mathcal{H}_{\beta, p} \rightarrow \mathcal{H}_{\beta, p}$ is bounded and $\mathcal{H}_{\alpha, p}$ is reducing for $T_{\beta, p}$ for each $\alpha \in \Lambda_\beta$ for a.e. $p \in X_\beta$. This implies that the map $\phi_{\alpha, \beta} : C^\ast_{\mathcal{E}_\beta, \text{DEC}} \left(\int^{\oplus}_{X_\beta} \mathcal{H}_{\beta, p} \, \mathrm{d} \mu_\beta \right) \rightarrow C^\ast_{\mathcal{E}_\alpha, \text{DEC}} \left(\int^{\oplus}_{X_\alpha} \mathcal{H}_{\alpha, p} \, \mathrm{d} \mu_\alpha \right)$ defined by 
\begin{equation} \label{eq; phi alpha beta}
\phi_{\alpha, \beta} \left (T = \int^\oplus_{X_\beta} T_{\beta, p} \, \mathrm{d} \mu_\beta \right ) := \int^\oplus_{X_\alpha} T_{\beta, p}\big|_{\mathcal{H}_{\alpha, p}} \, \mathrm{d} \mu_\alpha,
\end{equation}
is a well defined normal $\ast$-homomorphism, whenever $\alpha \in \Lambda_\beta$. Moreover, $\phi_{\alpha, \alpha} = \mathrm{Id}_{\int^{\oplus}_{X_\alpha} \mathcal{H}_{\alpha, p} \, \mathrm{d} \mu_\alpha}$ and $\phi_{\alpha, \beta} \circ \phi_{\beta, \gamma} = \phi_{\alpha, \gamma}$, whenever $\alpha \leq \beta \leq \gamma$. This shows that $\left (  \left \{ C^\ast_{\mathcal{E}_\alpha, \text{DEC}} \left(\int^{\oplus}_{X_\alpha} \mathcal{H}_{\alpha, p} \, \mathrm{d} \mu_\alpha \right) \right \}_{\alpha \in \Lambda}, \{ \phi_{\alpha, \beta} \}_{\alpha \leq \beta} \right )$ is a projective system of von Neumann algebras. 

Recall from (\ref{obs;M DEC and M DIAG are star algebras}) of Observations \ref{obs; 1} that $C^\ast_{\mathcal{E}, \text{DEC}} \left(\displaystyle \dilX \mathcal{D}_p \, \dmu \right)$ is a locally convex $\ast$-algebra with respect to the $C^\ast$-seminorm given by $T \mapsto \left  \| T\big|_{\mathcal{H}_\alpha}  \right \|$ for each $\alpha \in \Lambda$ (see Example \ref{ex;lca} and Equation \eqref{eq; norm of T restricted to H alpha}). In view of  Equations \eqref{eq;iso}, \eqref{eq;restriction of DecLBO}, we define a map $\phi_\alpha : C^\ast_{\mathcal{E}, \text{DEC}} \left(\displaystyle \dilX \mathcal{D}_p \, \dmu \right) \rightarrow   C^\ast_{\mathcal{E}_\alpha, \text{DEC}} \left(\int^{\oplus}_{X_\alpha} \mathcal{H}_{\alpha, p} \, \mathrm{d} \mu_\alpha \right)$ by 
\begin{equation} \label{eq; phi alpha}
\phi_\alpha \left ( \displaystyle \dilX T_p \, \dmu \right ) := \int^\oplus_{X_\alpha} T_{p} \big|_{\mathcal{H}_{\alpha, p}} \,  \mathrm{d} \mu_\alpha
\end{equation}
is a continuous $\ast$-homomorphism for each $\alpha \in \Lambda$. In addition, whenever $\alpha \leq \beta$, 
\begin{align*}
\phi_{\alpha, \beta} \circ \phi_\beta \left ( \displaystyle \dilX T_p \, \dmu \right ) = \phi_{\alpha, \beta} \left ( \int^\oplus_{X_\beta} T_{p} \big|_{\mathcal{H}_{\beta, p}} \,  \mathrm{d} \mu_\beta \right ) 
&= \int^\oplus_{X_\alpha} T_{p} \big|_{\mathcal{H}_{\alpha, p}} \,  \mathrm{d} \mu_\alpha \\
&= \phi_\alpha \left ( \displaystyle \dilX T_p \, \dmu \right ).
\end{align*}
This shows that the pair $ \left (C^\ast_{\mathcal{E}, \text{DEC}} \left(\displaystyle \dilX \mathcal{D}_p \, \dmu \right) , \{ \phi_\alpha \}_{\alpha \in \Lambda} \right )$ is compatible with the projective system $\left (  \left \{ C^\ast_{\mathcal{E}_\alpha, \text{DEC}} \left(\int^{\oplus}_{X_\alpha} \mathcal{H}_{\alpha, p} \, \mathrm{d} \mu_\alpha \right) \right \}_{\alpha \in \Lambda}, \{ \phi_{\alpha, \beta} \}_{\alpha \leq \beta} \right )$  of von Neumann algebras \big (for definitions, see Subsection 1.1 of \cite{AG} \big ). 

To prove that $ \left (C^\ast_{\mathcal{E}, \text{DEC}} \left(\displaystyle \dilX \mathcal{D}_p \, \dmu \right) , \{ \phi_\alpha \}_{\alpha \in \Lambda} \right )$ is the projective limit (in the category of topological algebras) of  $\left (  \left \{ C^\ast_{\mathcal{E}_\alpha, \text{DEC}} \left(\int^{\oplus}_{X_\alpha} \mathcal{H}_{\alpha, p} \, \mathrm{d} \mu_\alpha \right) \right \}_{\alpha \in \Lambda}, \{ \phi_{\alpha, \beta} \}_{\alpha \leq \beta} \right )$,  we consider a locally convex $\ast$-algebra $\mathcal{W}$ with a family $\{ \psi_\alpha \}_{\alpha \in \Lambda}$, where $\psi_\alpha : \mathcal{W} \rightarrow  C^\ast_{\mathcal{E}_\alpha, \text{DEC}} \left(\int^{\oplus}_{X_\alpha} \mathcal{H}_{\alpha, p} \, \mathrm{d} \mu_\alpha \right)$ is a continuous $\ast$-homomorphism for each $\alpha \in \Lambda$.  Let the pair $ \big (\mathcal{W} , \{ \psi_\alpha \}_{\alpha \in \Lambda} \big )$ be compatible with the projective system given by $\left (  \left \{ C^\ast_{\mathcal{E}_\alpha, \text{DEC}} \left(\int^{\oplus}_{X_\alpha} \mathcal{H}_{\alpha, p} \, \mathrm{d} \mu_\alpha \right) \right \}_{\alpha \in \Lambda}, \{ \phi_{\alpha, \beta} \}_{\alpha \leq \beta} \right )$, that is,
\begin{equation} \label{eqn; compatible for W and phi alpha}
\phi_{\alpha, \beta} \circ \psi_\beta = \psi_\alpha \; \; \text{whenever} \; \; \alpha \leq \beta.
\end{equation}
For each fixed $w \in \mathcal{W}$ and $\alpha \in \Lambda$, since $\psi_\alpha(w)$ is decomposable, we have $\psi_\alpha(w) = \int^\oplus_{X_\alpha}  \psi_\alpha(w)_p \, \mathrm{d} \mu_\alpha$ for some family $\left \{ \psi_\alpha(w)_p  \right \}_{p \in X_\alpha}$ of bounded operators on $\mathcal{H}_{\alpha, p}$ satisfying the property that $\mathcal{H}_{\delta, p}$ is reducing subspace for $\psi_\alpha(w)_p$ for each $\delta \in \Lambda_{\alpha}$. Whenever $\alpha \leq \beta$, for each $w \in \mathcal{W}$, by following Equation \eqref{eqn; compatible for W and phi alpha}, we obtain
\begin{align*}
\int^\oplus_{X_\alpha}  \psi_\alpha(w)_p \, \mathrm{d} \mu_\alpha = \psi_\alpha(w) &= \phi_{\alpha, \beta} \circ \psi_\beta(w) \\
&= \phi_{\alpha, \beta}  \left ( \int^\oplus_{X_\beta}  \psi_\beta(w)_p \, \mathrm{d} \mu_\beta \right )   \\
&= \int^\oplus_{X_\alpha}  \psi_\beta(w)_p\big|_{\mathcal{H}_{\alpha, p}} \, \mathrm{d} \mu_\alpha.
\end{align*}
Thus, there exists a measurable set $E^w_{\alpha, \beta} \subseteq X_\alpha$ such that $\mu_\alpha(E^w_{\alpha, \beta}) = 0$ and $ \psi_\beta(w)_p \big |_{\mathcal{H}_{\alpha, p}} \neq  \psi_\alpha(w)_p$ for every $p \in E^w_{\alpha, \beta}$. 

Let $E^w_\alpha := \bigcup\limits_{\beta \in \Lambda, \alpha \leq \beta} E^w_{\alpha, \beta}$ and $E^w :=  \bigcup\limits_{\alpha \in \Lambda} E^w_{\alpha}$. Since $\Lambda$ is countable, $\mu_\alpha(E^w_\alpha) = 0$ for each $\alpha$ and $E^w$ is a measurable set. Moreover, from Proposition \ref{prop;m}, we get $\mu(E^w) = 0$. On the other hand, if $\mu$ is a counting measure on $X$, then for any $\alpha \leq \beta$, the set $E^w_{\alpha, \beta} = \emptyset$ and thus $E^w = \emptyset$. So in both cases (that is, when $\Lambda$ is countable or $\mu$ is a counting measure on $X$), we consider the set $X \setminus E^w$ and without loss of generality we again denote it by $X$.
Next, consider a family $\left \{ T^w_{\alpha, p} : \mathcal{H}_{\alpha, p} \rightarrow \mathcal{H}_{\alpha, p} \; \; : \; \; \alpha \in \Lambda, \; p \in X \right \}$ given by
\begin{equation*}
T^w_{\alpha, p} := \begin{cases}
\psi_\alpha(w)_p, & \text{if} \;\; p \in X_\alpha ; \\
\psi_\beta(w)_p\big|_{\mathcal{H}_{\alpha, p}}, & \text{if} \;\; p \in X \setminus X_\alpha;
\end{cases}
\end{equation*}
for any $\beta \in \Lambda$ such $\alpha \leq \beta$ and $p \in X_\beta$. Here, note that the choice of $\beta \in \Lambda$  does not make any difference in the definition of $T^w_{\alpha, p}$, as we have $\psi_\beta(w)_p\big|_{\mathcal{H}_{\alpha, p}} = \psi_\gamma(w)_p\big|_{\mathcal{H}_{\alpha, p}}$, whenever $\alpha \leq \beta, \gamma$  and $p \in X_\beta$ also $p \in X_\gamma$. Using this, for each fixed $w \in \mathcal{W}$ and $p \in X$, define a locally bounded operator  $T^w_p \in C^\ast_{\mathcal{E}_{p}} (\mathcal{D}_{p})$ as $T^w_p := \varprojlim\limits_{\alpha \in \Lambda} T^w_{\alpha, p}$ (see Equation \eqref{eq; inverese limit of bounded operators} for this notation).  Now we define a continuous linear map $\Psi : \mathcal{W} \rightarrow C^\ast_{\mathcal{E}, \text{DEC}} \left(\displaystyle \dilX \mathcal{D}_p \, \dmu \right)$ as
\begin{equation*}
\Psi(w) := \displaystyle \dilX T^w_p  \, \dmu \; \; \; \text{for each} \; \; w \in \mathcal{W}.
\end{equation*}
Next, for $w \in \mathcal{W}$ and $\alpha \in \Lambda$, by following Equation \eqref{eq; phi alpha}, we obtain 
\begin{equation*}
\phi_\alpha \circ \Psi(w) = \phi_\alpha \left ( \displaystyle \dilX T^w_p  \, \dmu \right ) = \int^\oplus_{X_\alpha} T^w_p \big|_{\mathcal{H}_{\alpha, p}} \,  \mathrm{d} \mu_\alpha = \int^\oplus_{X_\alpha} \psi_\alpha(w)_p \,  \mathrm{d} \mu_\alpha = \psi_\alpha(w).
\end{equation*}
Since $w \in \mathcal{W}$ and $\alpha \in \Lambda$ were chosen arbitrarily, we get $\phi_\alpha \circ \Psi = \psi_\alpha$ for each $\alpha \in \Lambda$. Now it remains to show that such a map $\Psi$ is unique. Suppose there exists another continuous linear map $\hat{\Psi} : \mathcal{W} \rightarrow C^\ast_{\mathcal{E}, \text{DEC}} \left(\displaystyle \dilX \mathcal{D}_p \, \dmu \right)$ such that $\phi_\alpha \circ \hat{\Psi} = \psi_\alpha$ for all $\alpha \in \Lambda$. For $w \in \mathcal{W}$, let $\Psi(w) = \displaystyle \dilX T^w_p \, \dmu$ and $\hat{\Psi}(w) = \displaystyle \dilX \hat{T}^w_p \, \dmu$. Then for each fixed $\alpha \in \Lambda$, by following Equation \eqref{eq; phi alpha}, we get 
\begin{equation*}
 \int^\oplus_{X_\alpha} T^w_{p} \big|_{\mathcal{H}_{\alpha, p}} \,  \mathrm{d} \mu_\alpha = \phi_\alpha \circ \Psi(w) = \psi_\alpha(w) = \phi_\alpha \circ \hat{\Psi}(w)  = \int^\oplus_{X_\alpha} \hat{T}^w_{p} \big|_{\mathcal{H}_{\alpha, p}} \,  \mathrm{d} \mu_\alpha.
\end{equation*}
As the above equation holds for each $\alpha \in \Lambda$, we get $\Psi(w) = \hat{\Psi}(w)$. Since $w \in \mathcal{W}$ was chosen arbitrarily, we get  that the map $\Psi$ is unique. Therefore, by following the universal property given in \cite{NCP} by the uniqueness of the projective limit (see  \cite[Section 1.1]{AG}), we obtain
\begin{align*}
C^\ast_{\mathcal{E}, \text{DEC}} \left(\displaystyle \dilX \mathcal{D}_p \, \dmu \right) = \varprojlim\limits_{\alpha \in \Lambda} C^\ast_{\mathcal{E}_\alpha, \text{DEC}} \left(\int^{\oplus}_{X_\alpha} \mathcal{H}_{\alpha, p} \, \mathrm{d} \mu_\alpha \right).
\end{align*}
This proves Claim 3. 

\noindent
\textbf{Proof of (a):} By using Claim 1, Claim 3 and Definition \ref{def; lva 1}, we get that $C^\ast_{\mathcal{E}, \text{DEC}} \left(\displaystyle \dilX \mathcal{D}_p \, \dmu \right) $ is a locally von Neumann algebra.

\textbf{Claim 4:} $C^\ast_{\mathcal{E}, \text{DIAG}} \left(\displaystyle \dilX \mathcal{D}_p \, \dmu \right) = \varprojlim\limits_{\alpha \in \Lambda} C^\ast_{\mathcal{E}_\alpha, \text{DIAG}} \left(\int^{\oplus}_{X_\alpha} \mathcal{H}_{\alpha, p} \, \mathrm{d} \mu_\alpha \right).$

This proof follows on similar lines as the proof of Claim 3. However, we provide a detailed proof for the sake of completeness. Let $\int^\oplus_{X_\beta} T_{\beta, p} \, \mathrm{d} \mu_\beta \in  C^\ast_{\mathcal{E}_\beta, \text{DIAG}} \left(\int^{\oplus}_{X_\beta} \mathcal{H}_{\beta, p} \, \mathrm{d} \mu_\beta \right)$. Then $T_{\beta, p} = c_{\beta, p} \cdot \mathrm{Id}_{\mathcal{H}_{\beta,p}}$, where $c_{\beta, p} \in \mathbb{C}$ for a.e. $p \in X_\beta$. Let $\beta\in \Lambda$ be fixed and observe the fact that $C^\ast_{\mathcal{E}_\beta, \text{DIAG}} \left(\int^{\oplus}_{X_\beta} \mathcal{H}_{\beta, p} \, \mathrm{d} \mu_\beta \right) \subseteq C^\ast_{\mathcal{E}_\beta, \text{DEC}} \left(\int^{\oplus}_{X_\beta} \mathcal{H}_{\beta, p} \, \mathrm{d} \mu_\beta \right)$. Then by following Equation \eqref{eq; phi alpha beta}, we get
\begin{equation*}
\phi_{\alpha, \beta} \left (\int^\oplus_{X_\beta} c_{\beta, p} \cdot \mathrm{Id}_{\mathcal{H}_{\beta,p}} \, \mathrm{d} \mu_\beta \right ) =  \int^\oplus_{X_\alpha} c_{\beta, p} \cdot \mathrm{Id}_{\mathcal{H}_{\alpha, p}} \, \mathrm{d} \mu_\alpha, \; \; \; \text{whenever} \; \; \alpha \leq \beta.
\end{equation*}
This shows that $\phi_{\alpha, \beta} \left ( C^\ast_{\mathcal{E}_\beta, \text{DIAG}} \left(\int^{\oplus}_{X_\beta} \mathcal{H}_{\beta, p} \, \mathrm{d} \mu_\beta \right) \right ) \subseteq C^\ast_{\mathcal{E}_\alpha, \text{DIAG}} \left(\int^{\oplus}_{X_\alpha} \mathcal{H}_{\alpha, p} \, \mathrm{d} \mu_\alpha \right)$, whenever $\alpha \leq \beta$. Clearly $\left \{ \sigma_{\alpha, \beta} := \phi_{\alpha, \beta} \big |_{C^\ast_{\mathcal{E}_\beta, \text{DIAG}} \left(\int^{\oplus}_{X_\beta} \mathcal{H}_{\beta, p} \, \mathrm{d} \mu_\beta \right) } \right \}_{\alpha \leq \beta}$ is a family of  normal $\ast$-homomorphisms of von Neumann algebras with the property that $\sigma_{\alpha, \alpha} = \mathrm{Id}_{\int^{\oplus}_{X_\alpha} \mathcal{H}_{\alpha, p} \, \mathrm{d} \mu_\alpha}$ and $\sigma_{\alpha, \beta} \circ \sigma_{\beta, \gamma} = \sigma_{\alpha, \gamma}$, whenever $\alpha \leq \beta \leq \gamma$. This implies $\left (  \left \{ C^\ast_{\mathcal{E}_\alpha, \text{DIAG}} \left(\int^{\oplus}_{X_\alpha} \mathcal{H}_{\alpha, p} \, \mathrm{d} \mu_\alpha \right) \right \}_{\alpha \in \Lambda}, \{ \sigma_{\alpha, \beta} \}_{\alpha \leq \beta} \right )$ is a projective system of abelian von Neumann algebras.

Recall from (\ref{obs;M DEC and M DIAG are star algebras}) of Observations \ref{obs; 1} that  $C^\ast_{\mathcal{E}, \text{DIAG}} \left(\displaystyle \dilX \mathcal{D}_p \, \dmu \right)$ is a locally convex $\ast$-algebra. Also, we have $C^\ast_{\mathcal{E}, \text{DIAG}} \left(\displaystyle \dilX \mathcal{D}_p \, \dmu \right) \subseteq C^\ast_{\mathcal{E}, \text{DEC}} \left(\displaystyle \dilX \mathcal{D}_p \, \dmu \right)$ (see Equation \eqref{eqn; containment}). From Equation \eqref{eq; phi alpha}, it follows that $\phi_\alpha \left ( C^\ast_{\mathcal{E}, \text{DIAG}} \left(\displaystyle \dilX \mathcal{D}_p \, \dmu \right) \right ) = C^\ast_{\mathcal{E}_\alpha, \text{DIAG}} \left(\int^{\oplus}_{X_\alpha} \mathcal{H}_{\alpha, p} \, \mathrm{d} \mu_\alpha \right)$ for each $\alpha \in \Lambda$. As a consequence, for each $\alpha \in \Lambda$, the map 
\begin{equation*}
\sigma_\alpha := \phi_\alpha \big |_{C^\ast_{\mathcal{E}, \text{DIAG}} \left(\displaystyle \dilX \mathcal{D}_p \, \dmu \right)} : C^\ast_{\mathcal{E}, \text{DIAG}} \left(\displaystyle \dilX \mathcal{D}_p \, \dmu \right) \rightarrow C^\ast_{\mathcal{E}_\alpha, \text{DIAG}} \left(\int^{\oplus}_{X_\alpha} \mathcal{H}_{\alpha, p} \, \mathrm{d} \mu_\alpha \right)    
\end{equation*}
is a continuous $\ast$-homomorphism. Moreover, for any $T \in C^\ast_{\mathcal{E}, \text{DIAG}} \left(\displaystyle \dilX \mathcal{D}_p \, \dmu \right)$ and for any $\alpha \leq \beta$, we have
\begin{equation*}
\sigma_{\alpha, \beta} \circ \sigma_\beta ( T ) = \phi_{\alpha, \beta} \circ \phi_\beta(T) = \phi_\alpha (T) = \sigma_\alpha(T).
\end{equation*}
This shows that the pair $ \left (C^\ast_{\mathcal{E}, \text{DIAG}} \left(\displaystyle \dilX \mathcal{D}_p \, \dmu \right), \{ \sigma_\alpha \}_{\alpha \in \Lambda} \right )$ is compatible with the projective system $\left (  \left \{ C^\ast_{\mathcal{E}_\alpha, \text{DIAG}} \left(\int^{\oplus}_{X_\alpha} \mathcal{H}_{\alpha, p} \, \mathrm{d} \mu_\alpha \right) \right \}_{\alpha \in \Lambda}, \{ \sigma_{\alpha, \beta} \}_{\alpha \leq \beta} \right )$  of abelian von Neumann algebras \big (for definition, see Subsection 1.1 of \cite{AG} \big ). 

To prove that $ \left (C^\ast_{\mathcal{E}, \text{DIAG}} \left(\displaystyle \dilX \mathcal{D}_p \, \dmu \right) , \{ \sigma_\alpha \}_{\alpha \in \Lambda} \right )$ is the projective limit (in the category of topological algebras) of  $\left (  \left \{ C^\ast_{\mathcal{E}_\alpha, \text{DIAG}} \left(\int^{\oplus}_{X_\alpha} \mathcal{H}_{\alpha, p} \, \mathrm{d} \mu_\alpha \right) \right \}_{\alpha \in \Lambda}, \{ \sigma_{\alpha, \beta} \}_{\alpha \leq \beta} \right )$,  we consider a locally convex $\ast$-algebra $\mathcal{W}$ with a family $\{ \psi_\alpha \}_{\alpha \in \Lambda}$, where $\psi_\alpha : \mathcal{W} \rightarrow  C^\ast_{\mathcal{E}_\alpha, \text{DIAG}} \left(\int^{\oplus}_{X_\alpha} \mathcal{H}_{\alpha, p} \, \mathrm{d} \mu_\alpha \right)$ is a continuous $\ast$-homomorphism for each $\alpha \in \Lambda$.  Let the pair $ \big (\mathcal{W} , \{ \psi_\alpha \}_{\alpha \in \Lambda} \big )$ be compatible with the projective system given by $\left (  \left \{ C^\ast_{\mathcal{E}_\alpha, \text{DIAG}} \left(\int^{\oplus}_{X_\alpha} \mathcal{H}_{\alpha, p} \, \mathrm{d} \mu_\alpha \right) \right \}_{\alpha \in \Lambda}, \{ \sigma_{\alpha, \beta} \}_{\alpha \leq \beta} \right )$, that is,
\begin{equation} \label{eqn; compatible for W and sigma alpha}
\sigma_{\alpha, \beta} \circ \psi_\beta = \psi_\alpha \; \; \text{whenever} \; \; \alpha \leq \beta.
\end{equation}
For each fixed $w \in \mathcal{W}$ and $\alpha \in \Lambda$, since $\psi_\alpha(w)$ is diagonalizable, we have $\psi_\alpha(w) = \int^\oplus_{X_\alpha}  \psi_\alpha(w)_p \, \mathrm{d} \mu_\alpha$ for some family $\left \{ \psi_\alpha(w)_p \in  \mathbb{C} \cdot \mathrm{Id}_{\mathcal{H}_{\alpha,p}} \right \}_{p \in X_\alpha}$. Whenever $\alpha \leq \beta$, for each $w \in \mathcal{W}$, by following Equation \eqref{eqn; compatible for W and sigma alpha}, we obtain
\begin{align*}
\int^\oplus_{X_\alpha}  \psi_\alpha(w)_p \, \mathrm{d} \mu_\alpha = \psi_\alpha(w) &= \sigma_{\alpha, \beta} \circ \psi_\beta(w) \\
&= \sigma_{\alpha, \beta}  \left ( \int^\oplus_{X_\beta}  \psi_\beta(w)_p \, \mathrm{d} \mu_\beta \right )   \\
&= \int^\oplus_{X_\alpha}  \psi_\beta(w)_p\big|_{\mathcal{H}_{\alpha, p}} \, \mathrm{d} \mu_\alpha.
\end{align*}
Thus, there exists a measurable set $E^w_{\alpha, \beta} \subseteq X_\alpha$ such that $\mu_\alpha(E^w_{\alpha, \beta}) = 0$ and $ \psi_\beta(w)_p \big |_{\mathcal{H}_{\alpha, p}} \neq  \psi_\alpha(w)_p$ for every $p \in E^w_{\alpha, \beta}$. 

Let $E^w_\alpha := \bigcup\limits_{\beta \in \Lambda, \alpha \leq \beta} E^w_{\alpha, \beta}$ and $E^w :=  \bigcup\limits_{\alpha \in \Lambda} E^w_{\alpha}$. Since $\Lambda$ is countable, $\mu_\alpha(E^w_\alpha) = 0$ for each $\alpha$ and $E^w$ is a measurable set. Moreover, from Proposition \ref{prop;m}, we get $\mu(E^w) = 0$. On the other hand, if $\mu$ is a counting measure on $X$, then for any $\alpha \leq \beta$, the set $E^w_{\alpha, \beta} = \emptyset$ and thus $E^w = \emptyset$. So in both cases (that is, when $\Lambda$ is countable or $\mu$ is a counting measure on $X$), we consider the set $X \setminus E^w$ and without loss of generality we again denote it by $X$.
Next, consider a family $\left \{ T^w_{\alpha, p} \in \mathbb{C} \cdot \mathrm{Id}_{\mathcal{H}_{\alpha,p}} \; \; : \; \; \alpha \in \Lambda, \; p \in X \right \}$ given by
\begin{equation*}
T^w_{\alpha, p} := \begin{cases}
\psi_\alpha(w)_p, & \text{if} \;\; p \in X_\alpha ; \\
\psi_\beta(w)_p\big|_{\mathcal{H}_{\alpha, p}}, & \text{if} \;\; p \in X \setminus X_\alpha;
\end{cases}
\end{equation*}
for any $\beta \in \Lambda$ such $\alpha \leq \beta$ and $p \in X_\beta$. Here, note that the choice of $\beta \in \Lambda$  does not make any difference in the definition of $T^w_{\alpha, p}$, as we have $\psi_\beta(w)_p\big|_{\mathcal{H}_{\alpha, p}} = \psi_\gamma(w)_p\big|_{\mathcal{H}_{\alpha, p}}$, whenever $\alpha \leq \beta, \gamma$  and $p \in X_\beta$ also $p \in X_\gamma$. Using this, for each fixed $w \in \mathcal{W}$ and $p \in X$, define a locally bounded operator  $T^w_p \in C^\ast_{\mathcal{E}_{p}} (\mathcal{D}_{p})$ as $T^w_p := \varprojlim\limits_{\alpha \in \Lambda} T^w_{\alpha, p}$ (see Equation \eqref{eq; inverese limit of bounded operators}). In particular, we get $T^w_p \in \mathbb{C} \cdot \mathrm{Id}_{\mathcal{D}_{p}}$. Now we define a continuous linear map $\Psi : \mathcal{W} \rightarrow C^\ast_{\mathcal{E}, \text{DIAG}} \left(\displaystyle \dilX \mathcal{D}_p \, \dmu \right)$ as
\begin{equation*}
\Psi(w) := \displaystyle \dilX T^w_p  \, \dmu \; \; \; \text{for each} \; \; w \in \mathcal{W}.
\end{equation*}
Next, for $w \in \mathcal{W}$ and $\alpha \in \Lambda$, we obtain 
\begin{equation*}
\sigma_\alpha \circ \Psi(w) = \sigma_\alpha \left ( \displaystyle \dilX T^w_p  \, \dmu \right ) = \int^\oplus_{X_\alpha} T^w_p \big|_{\mathcal{H}_{\alpha, p}} \,  \mathrm{d} \mu_\alpha = \int^\oplus_{X_\alpha} \psi_\alpha(w)_p \,  \mathrm{d} \mu_\alpha = \psi_\alpha(w).
\end{equation*}
Since $w \in \mathcal{W}$ and $\alpha \in \Lambda$ were chosen arbitrarily, we get $\sigma_\alpha \circ \Psi = \psi_\alpha$ for each $\alpha \in \Lambda$. Now it remains to show that such a map $\Psi$ is unique. Suppose there exists another map $\hat{\Psi} : \mathcal{W} \rightarrow C^\ast_{\mathcal{E}, \text{DIAG}} \left(\displaystyle \dilX \mathcal{D}_p \, \dmu \right)$ such that for all $\alpha \in \Lambda$, we have $\sigma_\alpha \circ \hat{\Psi} = \psi_\alpha$. For $w \in \mathcal{W}$, let $\Psi(w) = \displaystyle \dilX T^w_p \, \dmu$ and $\hat{\Psi}(w) = \displaystyle \dilX \hat{T}^w_p \, \dmu$. Then for each fixed $\alpha \in \Lambda$, we get 
\begin{equation*}
 \int^\oplus_{X_\alpha} T^w_{p} \big|_{\mathcal{H}_{\alpha, p}} \,  \mathrm{d} \mu_\alpha = \sigma_\alpha \circ \Psi(w) = \psi_\alpha(w) = \sigma_\alpha \circ \hat{\Psi}(w)  = \int^\oplus_{X_\alpha} \hat{T}^w_{p} \big|_{\mathcal{H}_{\alpha, p}} \,  \mathrm{d} \mu_\alpha.
\end{equation*}
As the above equation holds for each $\alpha \in \Lambda$, we get $\Psi(w) = \hat{\Psi}(w)$. Since $w \in \mathcal{W}$ was chosen arbitrarily, we get  that the map $\Psi$ is unique. Therefore, by following the universal property given in \cite{NCP} by the uniqueness of the projective limit (see  \cite[Section 1.1]{AG}), we obtain
\begin{align*}
C^\ast_{\mathcal{E}, \text{DIAG}} \left(\displaystyle \dilX \mathcal{D}_p \, \dmu \right) = \varprojlim\limits_{\alpha \in \Lambda} C^\ast_{\mathcal{E}_\alpha, \text{DIAG}} \left(\int^{\oplus}_{X_\alpha} \mathcal{H}_{\alpha, p} \, \mathrm{d} \mu_\alpha \right).
\end{align*}
This proves Claim 4. 

\noindent
\textbf{Proof of (b):} By using Claim 2, Claim 4 and Definition \ref{def; lva 1}, we conclude that $C^\ast_{\mathcal{E}, \text{DIAG}} \left(\displaystyle \dilX \mathcal{D}_p \, \dmu \right)$ is an abelian locally von Neumann algebra. This completes the proof.
\end{proof}

We recall that for $\mathcal{M} \subseteq C^*_\mathcal{F}(\mathcal{D})$, where $\big \{ \mathcal{H}; \mathcal{F} = \{\mathcal{H}_{\alpha}\}_{\alpha \in \Lambda}; \mathcal{D} \big \}$  is a quantized domain, the commutant of $\mathcal{M}$ is denoted by $\mathcal{M}^\prime$ and is defined as 
\begin{equation} \label{eqn; commutant}
\mathcal{M}^\prime := \{ T \in C^*_\mathcal{F}(\mathcal{D}) ~~:~~ TS = ST ~~ \text{for all} ~~ S \in \mathcal{M} \}.    
\end{equation}
In the following remark, we obtain inclusion relations between locally von Neumann algebras described in Theorem \ref{thm;DEC and DIAG LvNA} and their commutants.

\begin{remark} \label{rem; commutants containments}
Let $\big(\Lambda, \leq \big)$ be a directed poset and $(X, \Sigma, \mu)$ be a locally measure space as mentioned in Note \ref{note; lms}. For each $p \in X$ assign a quantized domain $\big \{ \mathcal{H}_p; \mathcal{E}_p = \{\mathcal{H}_{\alpha, p}\}_{\alpha \in \Lambda}; \mathcal{D}_p \big \}$. Then we get the following containments
\begin{align*}
 C^\ast_{\mathcal{E}, \text{DEC}} \left(\displaystyle \dilX \mathcal{D}_p \, \dmu \right) &\subseteq  \left (C^\ast_{\mathcal{E}, \text{DIAG}} \left(\displaystyle \dilX \mathcal{D}_p \, \dmu \right) \right)^\prime ;  \\
 C^\ast_{\mathcal{E}, \text{DIAG}} \left(\displaystyle \dilX \mathcal{D}_p \, \dmu \right) &\subseteq \left (C^\ast_{\mathcal{E}, \text{DEC}} \left(\displaystyle \dilX \mathcal{D}_p \, \dmu \right) \right)^\prime.
\end{align*}
To see this, let $T \in C^\ast_{\mathcal{E}, \text{DEC}} \left(\displaystyle \dilX \mathcal{D}_p \, \dmu \right)$ and $S \in C^\ast_{\mathcal{E}, \text{DIAG}} \left(\displaystyle \dilX \mathcal{D}_p \, \dmu \right)$. Then by following Definition \ref{def;DecDiag(lbo)} we get a family $ \big \{ T_p \in C^\ast_{\mathcal{E}_p} \left(\mathcal{D}_p \right) \big \}_{p \in X}$ of locally bounded operators and a measurable function $f : X \rightarrow \mathbb{C}$ such that for any $u \in \displaystyle \dilX \mathcal{D}_p \, \dmu$, we have
\begin{align*}
(Tu)(p) = T_pu(p)  \; \; \text{and}  \; \;  (Su)(p) = f(p)u(p) \; \; \text{for a.e.} \; \; p \in X.
\end{align*}
For every $u = \displaystyle \dilX u(p) \, \dmu \in \displaystyle \dilX \mathcal{D}_p \, \dmu$, we have
\begin{align*}
\big (TS \big ) \left (\dilX u(p) \, \dmu \right ) &= T \left(\dilX f(p)u(p) \, \dmu \right) \\
&= \dilX T_p f(p) u(p) \, \dmu \\
&= \dilX  f(p) T_p u(p) \, \dmu \\
&= S \left (\dilX T_pu(p) \, \dmu \right) \\
&= \big  (ST \big ) \left(\dilX u(p) \, \dmu \right )
\end{align*}
Since $T$ and $S$ were arbitrarily chosen, we obtain the desired containments.  
\end{remark}

Motivated by the result described in Theorem \ref{thm;DeDibo vNA}, under certain conditions, we show that 
\begin{equation*}
C^\ast_{\mathcal{E}, \text{DEC}} \left(\displaystyle \dilX \mathcal{D}_p \, \dmu \right) = \left ( C^\ast_{\mathcal{E}, \text{DIAG}} \left(\displaystyle \dilX \mathcal{D}_p \, \dmu \right) \right )^ \prime.   
\end{equation*}
To prove this result, we need the following two lemmas.


\begin{lemma} \label{lem; V_alphaTV_alpha star is decomposable}
Let $\big(\Lambda, \leq \big)$ be a directed poset and $(X, \Sigma, \mu)$ be a locally measure space as mentioned in Note \ref{note; lms}. For each $p \in X$ assign a quantized domain $\big \{ \mathcal{H}_p; \mathcal{E}_p = \{\mathcal{H}_{\alpha, p}\}_{\alpha \in \Lambda}; \mathcal{D}_p \big \}$. Suppose $T \in \left (C^\ast_{\mathcal{E}, \text{DIAG}} \left(\displaystyle \dilX \mathcal{D}_p \, \dmu \right) \right)^\prime$, then the bounded operator $V_\alpha T V_\alpha^\ast$  (where, $V_\alpha : \mathcal{H}_\alpha \rightarrow \int^\oplus_{X_\alpha} \mathcal{H}_{\alpha, p} \, \mathrm{d} \mu_\alpha$ is defined as in Equation \eqref{eq;iso})  on $\int^\oplus_{X_\alpha} \mathcal{H}_{\alpha, p} \, \mathrm{d} \mu_\alpha$ is decomposable for each $\alpha \in \Lambda$.
\end{lemma}
\begin{proof}
Suppose $T \in \left ( C^\ast_{\mathcal{E}, \text{DIAG}} \left(\displaystyle \dilX \mathcal{D}_p \, \dmu \right) \right )^ \prime$. Consider the unitary operator $V_\alpha$ defined in Equation \eqref{eq;iso}. To show $V_\alpha T V_\alpha^\ast$ is a decomposable bounded operator on $\int^\oplus_{X_\alpha} \mathcal{H}_{\alpha, p} \, \mathrm{d} \mu_\alpha$, it is enough to prove $V_\alpha T V_\alpha^\ast$ commutes with every diagonalizable bounded operator on $\int^\oplus_{X_\alpha} \mathcal{H}_{\alpha, p} \, \mathrm{d} \mu_\alpha$ (see Theorem \ref{thm;DeDibo vNA}). Let $\int^\oplus_{X_\alpha} c_{\alpha, p} \cdot \mathrm{Id}_{\mathcal{H}_{\alpha,p}}  \, \mathrm{d} \mu_\alpha$ be an arbitrary diagonalizable bounded operator and $\int^\oplus_{X_\alpha} x(p)  \, \mathrm{d} \mu_\alpha \in \int^\oplus_{X_\alpha} \mathcal{H}_{\alpha,p}  \, \mathrm{d} \mu_\alpha$. Then 
\begin{equation*}
f(p) := \begin{cases}
c_{\alpha, p}, & \text{if} \;\; p \in X_\alpha; \\
o & \text{otherwise}.
\end{cases} 
\end{equation*}  
is a measurable function on $X$ and $u_x \in \mathcal{H}_\alpha \subseteq \displaystyle \dilX \mathcal{D}_p \, \dmu$ ($u_x$ defined as in Equation \eqref{eq; ux defined with x}). Now using the fact that $V_{\alpha}^{\ast}\left(\int^\oplus_{X_\alpha} x(p)  \, \mathrm{d} \mu_\alpha \right) = u_{x}$ and  $\left (\int^\oplus_{X_\alpha} c_{\alpha, p}  \, \mathrm{d} \mu_\alpha \right ) V_\alpha = V_\alpha \displaystyle \dilX f(p) \cdot \mathrm{Id}_{\mathcal{D}_p}  \, \mathrm{d} \mu,$ we get
\begin{align*}
\left ( V_\alpha T V^\ast_\alpha \right) \left (  \int^\oplus_{X_\alpha} c_{\alpha, p} \cdot \mathrm{Id}_{\mathcal{H}_{\alpha,p}}  \, \mathrm{d} \mu_\alpha \right ) \int^\oplus_{X_\alpha} x(p)  \, \mathrm{d} \mu_\alpha  &= \left ( V_\alpha T V^\ast_\alpha \right) \int^\oplus_{X_\alpha}  c_{\alpha, p} \cdot x(p)  \, \mathrm{d} \mu_\alpha \\
&=  V_\alpha T \left( \dilX f(p) \cdot \mathrm{Id}_{\mathcal{D}_p}  \, \mathrm{d} \mu \right) u_{x} \\
&= V_\alpha \left( \dilX f(p) \cdot \mathrm{Id}_{\mathcal{D}_p}  \, \mathrm{d} \mu \right) T u_{x} \\
&=  \left(\int^\oplus_{X_\alpha} c_{\alpha, p} \cdot \mathrm{Id}_{\mathcal{H}_{\alpha,p}} \right) \left(V_{\alpha}TV_{\alpha}^{\ast} \right) \int^\oplus_{X_\alpha} x(p)  \, \mathrm{d} \mu_\alpha.\\ 
\end{align*}
Therefore, $V_\alpha T V_\alpha^\ast$ is decomposable for every $\alpha \in \Lambda$.
\end{proof}

\begin{lemma}
Following the notions of Lemma \ref{lem; V_alphaTV_alpha star is decomposable}, if $T \in C^\ast_\mathcal{E} \left ( \displaystyle \dilX \mathcal{D}_p \, \dmu  \right )$, then the family $\big \{ V_\alpha T V^\ast_\alpha \big \}_{\alpha \in \Lambda}$ 
 of bounded operators satisfies the following relation,
\begin{equation*}
V_\alpha T^n V_\alpha^\ast = \big ( V_\alpha J^\ast_{\beta, \alpha} V_\beta^\ast \big ) \big ( V_\beta T^n V_\beta^\ast \big ) \big ( V_\beta J_{\beta, \alpha} V_\alpha^\ast \big ),
\end{equation*}
\text{whenever}\; $\alpha \leq \beta$. Here $J_{\beta, \alpha}\colon \mathcal{H}_{\alpha} \to \mathcal{H}_{\beta}$ is the inclusion map (for $\alpha \leq \beta).$
\end{lemma}
\begin{proof}
Recall that $\displaystyle \dilX \mathcal{D}_p \, \dmu = \varinjlim\limits_{\alpha \in \Lambda} \mathcal{H}_\alpha$, where $\mathcal{H}_
\alpha$ is defined as in Equation \eqref{eq; H alpha}. Let $x \in \int^\oplus_{X_\alpha} \mathcal{H}_{\alpha, p} \, \mathrm{d} \mu_\alpha$ and $\alpha \leq \beta$. By using the definitions of $V_\alpha$ and $V_\beta$ (see Equation \eqref{eq;iso}), we have 
\begin{equation} \label{eq; V beta J beta alpha v alpha star}
V_\beta J_{\beta, \alpha} V_\alpha^\ast(x) = \begin{cases}
x(p), & \text{if} \;\; p \in X_\alpha; \\
0_{\mathcal{H}_{\beta, p}} & \text{if} \;\; p \in X_\beta \setminus X_\alpha.
\end{cases}
\end{equation}
Now as $V_\alpha$, $V_\beta$ are unitary operators and $J_{\beta, \alpha}$ is an inclusion map, we get that $V_\beta J_{\beta, \alpha} V_\alpha^\ast$ is an isometry. Now by using the fact that $T$ is a locally bounded operator and $V_\beta$ is unitary, we obtain
\begin{align*}
\big ( V_\alpha T V_\alpha^\ast \big ) \big  (x \big ) = \big ( V_\alpha T \big ) \big ( V_\alpha^\ast x \big ) 
&= \big ( V_\alpha T J_{\beta, \alpha} \big ) \big ( V_\alpha^\ast x \big ) \\
&= \big ( V_\alpha J^\ast_{\beta, \alpha} T J_{\beta, \alpha} V_\alpha^\ast \big ) \big ( x \big ) \\
&= \big ( V_\alpha J^\ast_{\beta, \alpha} V_\beta^\ast \big ) \big ( V_\beta T V_\beta^\ast \big ) \big ( V_\beta J_{\beta, \alpha} V_\alpha^\ast \big ) \big ( x \big ).
\end{align*}
Since $x \in \int^\oplus_{X_\alpha} \mathcal{H}_{\alpha, p} \, \mathrm{d} \mu_\alpha$ was arbitrarily chosen, we get 
\begin{equation} \label{eq; dilation}
V_\alpha T V_\alpha^\ast = \big ( V_\alpha J^\ast_{\beta, \alpha} V_\beta^\ast \big ) \big ( V_\beta T V_\beta^\ast \big ) \big ( V_\beta J_{\beta, \alpha} V_\alpha^\ast \big ).
\end{equation}
Equivalently, the following diagram commutes.
\begin{center}
\begin{tikzcd}[sep=huge]
& \int^\oplus_{X_\beta} \mathcal{H}_{\beta, p} \, \mathrm{d} \mu_\beta \arrow[rr, "V_\beta T V_\beta^\ast"]  & & \int^\oplus_{X_\beta} \mathcal{H}_{\beta, p} \, \mathrm{d} \mu_\beta \\
& & \mathcal{H}_\beta  \arrow[ul, dotted, "V_\beta"] \arrow[ur, dotted, "V_\beta"']  & \\
& & \mathcal{H}_\alpha \arrow[u, dotted, "J_{\beta, \alpha}"] & \\
& \int^\oplus_{X_\alpha} \mathcal{H}_{\alpha, p} \, \mathrm{d} \mu_\alpha \arrow[ur, dotted, "V_\alpha^\ast"]  \arrow[uuu, "V_\beta J_{\beta, \alpha}V^\ast_\alpha"] \arrow[rr, "V_\alpha T V_\alpha^\ast"'] & &  \int^\oplus_{X_\alpha} \mathcal{H}_{\alpha, p} \, \mathrm{d} \mu_\alpha \arrow[ul, dotted, "V_\alpha^\ast"'] \arrow[uuu, "V_\beta J_{\beta, \alpha}V^\ast_\alpha"']
\end{tikzcd}
\end{center}
\end{proof}

\begin{theorem} \label{thm; M DEC = M DIAG Commutant}
Let $(\Lambda, \leq)$ be a directed poset and $(X, \Sigma, \mu)$ be a locally measure space (see Note \ref{note; lms}). For each $p \in X$ assign a quantized domain $\big \{ \mathcal{H}_p; \mathcal{E}_p = \{\mathcal{H}_{\alpha, p}\}_{\alpha \in \Lambda}; \mathcal{D}_p \big \}$. Suppose either $\Lambda$ is a countable set or $\mu$ is a countable measure on $X$, then 
\begin{equation*}
C^\ast_{\mathcal{E}, \text{DEC}} \left(\displaystyle \dilX \mathcal{D}_p \, \dmu \right) = \left (C^\ast_{\mathcal{E}, \text{DIAG}} \left(\displaystyle \dilX \mathcal{D}_p \, \dmu \right)\right)^\prime.
\end{equation*}
\end{theorem}
\begin{proof}
From Remark \ref{rem; commutants containments}, we know that $C^\ast_{\mathcal{E}, \text{DEC}} \left(\displaystyle \dilX \mathcal{D}_p \, \dmu \right) \subseteq \left (C^\ast_{\mathcal{E}, \text{DIAG}} \left(\displaystyle \dilX \mathcal{D}_p \, \dmu \right) \right)^\prime$. Thus, it is enough to show the other containment. Suppose $T \in \left (C^\ast_{\mathcal{E}, \text{DIAG}} \left(\displaystyle \dilX \mathcal{D}_p \, \dmu \right) \right)^\prime$, then by Lemma \ref{lem; V_alphaTV_alpha star is decomposable}, $V_\alpha T V_\alpha^\ast$ is a decomposable operator on $\int^\oplus_{X_\alpha} \mathcal{H}_{\alpha, p} \, \mathrm{d} \mu_\alpha$ for each $\alpha \in \Lambda$. Following (\ref{def;Decbo}) of Definition \ref{def;Debo}, for each $\alpha \in \Lambda$, there is a family $\big \{ S_{\alpha, p} \in \mathcal{B} \big ( \mathcal{H}_{\alpha, p} \big ) \big \}_{p \in X_\alpha}$  of bounded operators such that 
\begin{equation} \label{eq; V alpha T V alpha star}
V_\alpha T V_\alpha^\ast = \int^\oplus_{X_\alpha} S_{\alpha, p} \, \mathrm{d} \mu_\alpha.
\end{equation}
By following the same procedure for $\alpha \leq \beta$, we get $V_\beta T V_\beta^\ast = \int^\oplus_{X_\beta} S_{\beta, p} \, \mathrm{d} \mu_\beta$. Now, we claim that $S_{\beta, p} \big |_{\mathcal{H}_{\alpha, p}} = S_{\alpha, p}$ for a.e. $p \in X_\alpha$, whenever $\alpha \leq \beta$. Suppose there exists a measurable set $E_{\alpha, \beta} \subseteq X_\alpha$ such that $0 < \mu_\alpha(E_{\alpha, \beta}) < \infty$ (use $\sigma$-compactness of $X_\alpha$) and $S_{\beta, p} \big |_{\mathcal{H}_{\alpha, p}} \neq S_{\alpha, p}$ for every $p \in E_{\alpha, \beta}$. This implies that there exists a family $\big \{ \xi_{\alpha, p} \in \mathcal{H}_{\alpha, p} \big \}_{p \in E_{\alpha, \beta}}$ of vectors such that 
\begin{equation*}
S_{\beta, p}(\xi_{\alpha, p}) \neq S_{\alpha, p}(\xi_{\alpha, p}) \; \; \; \text{for every} \; \; p \in E_{\alpha, \beta}.
\end{equation*}
Without loss of generality, we assume that the family $\big \{ \xi_{\alpha, p}  \in \mathcal{H}_{\alpha, p} \big \}_{p \in E_{\alpha, \beta}}$ consists of unit vectors.
Now consider a family $\big \{ \hat{\xi}_{\alpha, p} \in \mathcal{H}_{\alpha, p} \big \}_{p \in X_\alpha}$ of vectors, where
\begin{equation*}
\hat{\xi}_{\alpha, p} = \begin{cases}
\xi_{\alpha, p}, & \text{if} \;\; p \in E_{\alpha, \beta}; \\
0_{\mathcal{H}_{\alpha, p}} & \text{if} \;\; p \in X_\alpha \setminus E_{\alpha, \beta}.
\end{cases}
\end{equation*}
Consider a map $x : X_\alpha \rightarrow \bigcup\limits_{p \in X_\alpha} \mathcal{H}_{\alpha, p}$ given by $x(p) := \hat{\xi}_{\alpha, p}$ for every $p \in X_\alpha$. Now, we use (2) of Definition \ref{def;dihs} to show $x \in \int^\oplus_{X_\alpha} \mathcal{H}_{\alpha, p} \, \mathrm{d} \mu_\alpha$. For any $y \in \int^\oplus_{X_\alpha} \mathcal{H}_{\alpha, p} \, \mathrm{d} \mu_\alpha$, the map $p \mapsto \|y(p)\|$ is in $\text{L}^{2}(X_{\alpha}, \mu_{\alpha})$ and
\begin{equation*}
\int_{X_\alpha} \; \big | \la x(p), y(p)  \ra \big| \; \mathrm{d} \mu_\alpha(p) = \int\limits_{E_{\alpha, \beta}} \; \big | \la \xi_{\alpha, p}, y(p)  \ra \big| \; \mathrm{d} \mu_\alpha(p) 
\leq \int\limits_{E_{\alpha, \beta}} \; \big \| y(p) \big \|  \; \mathrm{d} \mu_\alpha(p) < \infty.
\end{equation*}
The last inequality holds true as the map $p \mapsto \big \| 
y(p) \big \|_{\mathcal{H}_{{\alpha_u}, p}}$ when restricted to $E_{\alpha, \beta}$ is in $\text{L}^1(X_{\alpha}, \mu_{\alpha})$, as one may see that $\mu_{\alpha}(E_{\alpha, \beta}) < \infty$. This shows that $x \in \int^\oplus_{X_\alpha} \mathcal{H}_{\alpha, p} \, \mathrm{d} \mu_\alpha$.
Finally, by using Equation \eqref{eq; dilation} and Equation \eqref{eq; V beta J beta alpha v alpha star}, we get
\begin{align*}
\int^\oplus_{X_\alpha} S_{\alpha, p} x(p) \, \mathrm{d} \mu_\alpha &= \big ( V_\alpha T V_\alpha^\ast \big ) (x)\\
&= \big ( V_\alpha J^\ast_{\beta, \alpha} V_\beta^\ast \big ) \big ( V_\beta T V_\beta^\ast \big ) \big ( V_\beta J_{\beta, \alpha} V_\alpha^\ast \big ) (x)\\
&= \big ( V_\alpha J^\ast_{\beta, \alpha} V_\beta^\ast \big ) \big ( V_\beta T V_\beta^\ast \big ) \left ( \int^\oplus_{X_\beta} \big ( V_\beta J_{\beta, \alpha} V_\alpha^\ast \big ) (x) (p) \, \mathrm{d} \mu_\beta \right) \\
&= \big ( V_\alpha J^\ast_{\beta, \alpha} V_\beta^\ast \big ) \left ( \int^\oplus_{X_\beta} S_{\beta, p} \big ( V_\beta J_{\beta, \alpha} V_\alpha^\ast \big ) (x) (p) \, \mathrm{d} \mu_\beta \right) \\
&= \int^\oplus_{X_\alpha} S_{\beta, p} \big ( V_\beta J_{\beta, \alpha} V_\alpha^\ast \big ) (x) (p) \, \mathrm{d} \mu_\alpha \\
&= \int^\oplus_{X_\alpha} S_{\beta, p} x(p) \, \mathrm{d} \mu_\alpha.
\end{align*}
Thus $S_{\beta, p} x(p) = S_{\alpha, p} x(p)$ for a.e. $p \in X_\alpha$. In particular, for a.e. $p \in E_{\alpha, \beta}$, we have
\begin{equation*}
S_{\beta, p}(\xi_{\alpha, p}) = S_{\beta, p}(\hat{\xi}_{\alpha, p}) = S_{\beta, p}(x(p)) =  S_{\alpha, p}(x(p)) = S_{\alpha, p}(\hat{\xi}_{\alpha, p}) = S_{\alpha, p}(\xi_{\alpha, p}).
\end{equation*}
This is a contradiction. This implies that $\mu_\alpha(E_{\alpha, \beta}) = 0$ and hence for a.e. $p \in X_\alpha$
\begin{equation*}
S_{\beta, p} \big |_{\mathcal{H}_{\alpha, p}} = S_{\alpha, p}, \; \; \; \text{whenever} \; \; \alpha \leq \beta.
\end{equation*}

Let $E_\alpha := \bigcup\limits_{\beta \in \Lambda, \alpha \leq \beta} E_{\alpha, \beta}$ and $E :=  \bigcup\limits_{\alpha \in \Lambda} E_{\alpha}$. Since $\Lambda$ is countable, $\mu_\alpha(E_\alpha) = 0$ for each $\alpha$ and $E$ is a measurable set. Moreover, from Proposition \ref{prop;m}, we get $\mu(E) = 0$. On the other hand, if $\mu$ is a counting measure on $X$, then for any $\alpha \leq \beta$, the set $E_{\alpha, \beta} = \emptyset$ and thus $E = \emptyset$. So in both cases (that is, when $\Lambda$ is countable or $\mu$ is a counting measure on $X$), we consider the set $X \setminus E$. Without loss of generality we again denote it by $X$. Then consider the family $\big \{ T_{\alpha, p} \; : \; \alpha \in \Lambda, \, p \in X \big\}$ of bounded operators defined by 
\begin{equation} \label{eq; T alpha p}
T_{\alpha, p} := \begin{cases}
S_{\alpha, p}, & \text{if} \;\; p \in X_\alpha; \\
S_{\beta, p}\big|_{\mathcal{H}_{\alpha, p}} & \text{if} \;\; p \in X\setminus X_\alpha;
\end{cases}
\end{equation}
for any $\beta \in \Lambda$ such $\alpha \leq \beta$ and $p \in X_\beta$. Here, note that the choice of $\beta \in \Lambda$  does not make any difference in the definition of $T_{\alpha, p}$, as we have $S_{\beta, p}\big|_{\mathcal{H}_{\alpha, p}} = S_{\gamma, p}\big|_{\mathcal{H}_{\alpha, p}}$, whenever $\alpha \leq \beta, \gamma$  and $p \in X_\beta \cap X_\gamma$. Thus, for each fixed $p \in X$, the family $\big \{  T_{\alpha, p}  \big \}_{\alpha \in \Lambda}$ is such that $T_{\beta, p} \big |_{\mathcal{H}_{\alpha, p}} = T_{\alpha, p}$, whenever $\alpha \leq \beta$. This yields a locally bounded operator $T_p : \mathcal{D}_p \rightarrow \mathcal{D}_p$ (see Equation \eqref{eq; inverese limit of bounded operators}) given by
\begin{align*} 
T_p := \varprojlim\limits_{\alpha \in \Lambda} T_{\alpha, p}.
\end{align*}
Finally, we show that $T = \displaystyle \dilX T_p \, \dmu$. Let $u \in \displaystyle \dilX \mathcal{D}_p \, \dmu$. Then $u \in \mathcal{H}_\alpha$ for some $\alpha \in \Lambda$. This implies $Tu \in \mathcal{H}_\alpha$. By following the definition of $V_\alpha$,  Equation \eqref{eq; V alpha T V alpha star} and Equation \eqref{eq; T alpha p}, we get
\begin{equation*}
\big (Tu \big )(p) = \big (V_\alpha T u \big )(p) = \big (V_\alpha T V^\ast_\alpha \big )\big (V_\alpha u \big )(p) = S_{\alpha, p} u(p) =  T_{\alpha, p} u(p) = T_p u(p), \; \; \; \; \text{a.e.} \; \; p \in X_\alpha
\end{equation*}
and for $p \in X \setminus X_\alpha$, we have $u(p) = 0_{\mathcal{D}_p} = Tu(p) = T_pu(p)$. Since $u \in \displaystyle \dilX \mathcal{D}_p \, \dmu$ was chosen arbitrarily, by following  (\ref{def;Dec(lbo)}) of Definition \ref{def;DecDiag(lbo)}, we obtain 
$$T = \displaystyle \dilX T_p \, \dmu.$$ Thus $T \in C^\ast_{\mathcal{E}, \text{DEC}} \left(\displaystyle \dilX \mathcal{D}_p \, \dmu \right)$.  This proves the result. 
\end{proof}

\subsection*{Acknowledgment}
The first named author kindly acknowledges the financial support received as an Institute postdoctoral fellowship from the Indian Institute of Science Education and Research Mohali. The second named author would like to thank SERB (India) for a financial support in the form of a Startup Research Grant (File No. SRG/2022/001795). The authors express their sincere thanks to DST for a financial support in the form of the FIST grant (File No. SR/FST/MS-I/2019/46(C)) and the Department of Mathematical Sciences, IISER Mohali, for providing the necessary facilities to carry out this work.

\subsection*{Declaration} The authors declare that there are no conflicts of interest.

\bibliographystyle{plain}

\end{document}